\documentclass[a4paper,10pt,reqno]{amsart}
\usepackage[czech,english]{babel}
\usepackage{amsmath,amssymb,amsthm,amscd}
\usepackage[all]{xy}
\usepackage{hyperref}
\renewcommand{\iff}{if and only if }
\newcommand{\st}{such that }

\newcommand{\RMod}{R\hbox{{\rm -Mod}}}
\newcommand{\RoneMod}{R_1\hbox{{\rm -Mod}}}
\newcommand{\ModR}{\hbox{{\rm Mod-}}R}

\newcommand{\Ch}[1]{\mathrm{C}(#1)}
\newcommand{\Chac}[1]{\mathrm{C_{ac}}(#1)}
\newcommand{\Chtac}[1]{\mathrm{C_{tac}}(#1)}
\newcommand{\ChR}{\mathrm{Ch(\it R)}}

\newcommand{\Cogen}{\mathrm{Cogen}}

\newcommand{\FL}{\mathcal{FL}}
\newcommand{\PGF}{\mathcal{PGF}}

\newcommand{\EC}{\mathcal{EC}}
\newcommand{\RE}{\mathcal{RE}}
\newcommand{\Z}{\mathbb{Z}}
\newcommand{\Q}{\mathbb{Q}}
\newcommand{\Rl}{\mathbb{R}}

\DeclareMathOperator{\Hom}{Hom}

\DeclareMathOperator{\End}{End}
\DeclareMathOperator{\Ext}{Ext}
\DeclareMathOperator{\Tor}{Tor}
\DeclareMathOperator{\Ker}{Ker}
\DeclareMathOperator{\Img}{Im}
\DeclareMathOperator{\Coker}{Coker}

\DeclareMathOperator{\pres}{pres}

\DeclareMathOperator{\cf}{cf}

\newcommand{\op}{\mathrm{op}}

\usepackage[usenames]{color}
\usepackage[normalem]{ulem}
\newcommand{\chge}[1]{\begin{color}{red}#1\end{color}}
\newcommand{\out}[1]{\begin{color}{red}\sout{#1}\end{color}}
\renewcommand{\chge}[1]{#1}
\renewcommand{\out}[1]{}

\newcommand{\A}{\mathcal{A}}
\newcommand{\B}{\mathcal{B}}
\newcommand{\C}{\mathcal{C}}
\newcommand{\D}{\mathcal{D}}

\newcommand{\clS}{\mathcal{S}}

\theoremstyle{plain}
\newtheorem{thm}{Theorem}[section]
\newtheorem{prop}[thm]{Proposition}
\newtheorem{lem}[thm]{Lemma}

\newtheorem{cor}[thm]{Corollary}
\newtheorem{constr}[thm]{Construction}
\newtheorem{obser}[thm]{Observation}
\theoremstyle{definition}
\newtheorem{defn}[thm]{Definition}
\newtheorem{exm}[thm]{Example}

\theoremstyle{remark}
\newtheorem*{rem}{Remark}

\begin{document}
\title[$\Sigma$-cotorsion and Gorenstein projective modules]%
{Singular compactness and definability for $\Sigma$-cotorsion and Gorenstein modules}

\author{\textsc{Jan \v Saroch}}
\address{Charles University, Faculty of Mathematics and Physics, Department of Algebra \\ 
Sokolovsk\'{a} 83, 186 75 Praha~8, Czech Republic}
\email{saroch@karlin.mff.cuni.cz}

\author{\textsc{Jan \v S\v tov\'\i\v cek}}
\address{Charles University, Faculty of Mathematics and Physics, Department of Algebra \\ 
Sokolovsk\'{a} 83, 186 75 Praha~8, Czech Republic}
\email{stovicek@karlin.mff.cuni.cz}
 
\keywords{\chge{($\Sigma$)-cotorsion module, projectively coresolved Gorenstein flat module, Gorenstein injective module, complete cotorsion pair, singular compactness, infinite combinatorics}}

\thanks{\chge{The research of the authors has been supported by grant GA\v CR 17-23112S}}

\subjclass[2010]{\chge{16E30 (primary), 16B70, 03E75 (secondary)}}
\date{\today}

\begin{abstract} We introduce a general version of singular compactness theorem which makes it possible to show that being a $\Sigma$-cotorsion module is a property of the complete theory of the module. As an application of the powerful tools developed along the way, we give a new description of Gorenstein flat modules which implies that, regardless of the ring, the class of all Gorenstein flat modules forms the left-hand class of a perfect cotorsion pair. We also prove the dual result for Gorenstein injective modules.\end{abstract} 

\maketitle
\tableofcontents
\vspace{4ex}

\section*{Introduction}

\chge{
The aim of the paper is to establish new structural and approximation results about two types of homologically defined (and at least in the first case very well known) classes of modules:

\begin{enumerate}
\item Gorenstein flat and Gorenstein injective modules and
\item $\Sigma$-cotorsion modules.
\end{enumerate}  

What these seemingly distant classes of modules have in common is the rather non-obvious fact that one can learn deep facts about their structure using infinite combinatorics and set-theoretically flavored homological tools. This is despite the fact that the statements of the main results (Theorems~\ref{t:sigmacot}, \ref{t:GF} and \ref{t:GI} and their corollaries) are of purely module-theoretic and homological nature, and can be explained without any set theory. It is their proofs where infinite combinatorics plays crucial role, and the key ingredient brought by this paper is a new version of Shelah's singular compactness theorem for direct systems which do not necessarily consist of monomorphisms.

\smallskip

Gorenstein homological algebra, which is a version of relative homological algebra with roots on one hand in commutative algebra (and especially the celebrated Auslander--Buchsbaum formula) and on the other hand in modular representation theory of finite groups, has been developed for almost half a century; an interested reader may find a more detailed overview in the introduction of~\cite{XWC}. Our main result here is that, for \emph{any} ring, the classes of Gorenstein injective and Gorenstein flat modules sit in complete cotorsion pairs, the class of Gorenstein injective modules is enveloping and the class of Gorenstein flat modules is covering. In particular, any ring is GF-closed in the sense of~\cite{B}.

This contribution is perhaps best explained in the context of the new impetus which Gorenstein homological algebra recently got from the study of abelian model structures \cite{G-overview} and which allowed to import homotopical theoretic techniques. Since it was not known in general whether the standard classes of Gorenstein flat or injective modules had good approximation properties, Bravo, Hovey and Gillespie~\cite{BGH} were led to introduce a modification of the definitions of these classes, to ensure the existence of the required approximations in this way. Our results can thus be summarized as that this change was not necessary: the classes of Gorenstein flat and injective modules have good approximation properties and induce abelian model structure on their own for every ring, regardless of how daunting the ring is. The model structures of \cite{BGH} can then be recovered as a localization (Bousfield localization at the level of model categories or triangulated localization at the level of their homotopy categories) of the model structures arising from the standard classes.

\smallskip

The class of $\Sigma$-cotorsion modules, on the other hand, was studied~\cite{BS_sigma-cot,G-AH_sigma-rings,G-AH_model-th} in an attempt to generalize model theoretic methods for modules to arbitrary additive finitely accessible category (in the terminology of~\cite{AR}; they are also known under the term locally finitely presented additive categories \cite{CB}). Every finitely accessible additive category is equivalent to the category $\mathcal{FL}$ of flat modules over a ring $R$ (possibly non-unital, but with enough idempotents) and admits a natural (pure) exact structure inherited from $\ModR$. Moreover, as a consequence of the solution to the Flat Cover Conjecture~\cite{BEE}, this exact structure has enough injective objects, which are precisely the flat and cotorsion $R$-modules. The main theme of~\cite{G-AH_survey,G-AH_indec} is that there are even enough indecomposable flat cotorsion modules in order to cogenerate $\mathcal{FL}$, so that one can go on and define the Ziegler spectrum for $\mathcal{FL}$ (at least as a set, the topology still has not been defined in general at the time of writing this paper).

A $\Sigma$-cotorsion module is one whose every direct sum of copies is cotorsion. Thus $\Sigma$-cotorsion modules generalize classical $\Sigma$-pure-injective modules, which are well behaved and characterized by chain conditions on definable subgroups.

Our main result here is that $\Sigma$-cotorsion modules are also characterized by a~version of chain conditions, but these are way more complicated. As a consequence, if $C$ is a $\Sigma$-cotorsion module, then any module in the smallest definable class (= first-order axiomatizable and closed under direct sums and summands) containing $C$ is $\Sigma$-cotorsion as well. This, in particular, shows that being $\Sigma$-cotorsion is not a property of a particular module, but rather of its first-order theory in the language of modules over a given ring. The notion of $\Sigma$-cotorsion module is therefore one where homological algebra, model theory and infinite combinatorics meet each other in a fascinating way.

\smallskip

As already mentioned, our results are based on a collection of methods involving homological algebra and infinite combinatorics (stationarity and Mittag-Leffler condition, almost-freeness, singular compactness), which have been thoroughly studied by several authors in the last two decades. We use these to treat the following general questions for a class $\B\subseteq\ModR$:

\begin{enumerate}
\item[(a)] Given a module $M$ such that $\Ext^1_R(M,\B)=0$, when can we write $M = \varinjlim M_i$, where $\Ext^1_R(M_i,\B)=0$ and all the $M_i$ are $\kappa$-presented for some fixed cardinal $\kappa$, independent of $M$?
\item[(b)] Conversely, suppose that $M = \varinjlim M_i$, where $\Ext^1_R(M_i,\B)=0$. When can we conclude that $\Ext^1_R(M,\B)=0$?
\end{enumerate}

Our general strategy to understand classes of modules of the form $\Ker\Ext^1_R(-,\B)$ is first to give a positive answer to (a), using non-trivial closure properties of $\B$ (e.g.\ under direct limits or direct sums). In the best cases we can reach $\kappa=\aleph_0$, which reduces our questions to countably presented modules, which are in general very well understood. Then we use (b) for a possibly larger class $\B'\supseteq\B$.

\smallskip

The paper is organized as follows. We first collect the essential tools of infinite combinatorics in homological algebra in Section~\ref{sec:tools}, including the novel Theorems~\ref{t:main1} and~\ref{t:regular_filt}.

In section Section~\ref{sec:sigmacot} we prove Theorem~\ref{t:sigmacot}, which says that $\Sigma$-cotorsionness is a~property of a first-order theory, and study the corresponding intricate chain conditions (Definition~\ref{d:S_C}) which generalize previously known special cases for countable rings~\cite[Theorem 12]{G-AH_model-th} and non-discrete valuation domains~\cite[Theorem 3.8]{BS_sigma-cot}.

In Section~\ref{sec:Gor}, we introduce a new class of projectively coresolved Gorenstein flat modules, which turns out to be a part of a complete hereditary cotorsion pair (Theorems~\ref{t:pgf} and~\ref{t:Kaplansky}). This allows us to prove that also Gorenstein flat modules are a part of a complete cotorsion pair (Theorem~\ref{t:GF} and Corollary~\ref{c:GF}) and to define two new Quillen equivalent abelian model structures.

In Section~\ref{sec:GI} we prove that also Gorenstein injectives sit in a complete cotorsion pair (Theorem~\ref{t:GI}) and can be used to define an abelian model structure, whose existence was previously known only for particular cases of rings \cite[Theorem~7.12]{Kr}.

Finally, the last Sections~\ref{sec:singset} and~\ref{sec:singmod} are devoted to giving a proof for our new version of singular compactness (which we use in the form of Lemma~\ref{l:singmod}). The paper is concluded by appendix which describes standard but somewhat technical and not trivial operations on direct systems, which we constantly use.
}

\section{Relative projectivity, injectivity and general tools}
\label{sec:tools}

Unless stated otherwise, by a module, we mean a right $R$-module where $R$ is an associative unital ring. We denote the class of all modules by $\ModR$ and the class of all left $R$-modules by $\RMod$. In fact, our techniques work equally well for modules over small preadditive categories (in the sense of~\cite[Appendix B]{JL}) or, equivalently, unitary $R$-modules $M$ (i.e. satisfying $MR = M$) over (associative) rings with enough idempotents (see \cite[\S I.3]{Hap} for precise definitions).


\chge{A \emph{cotorsion pair} is a pair $\mathfrak{C} = (\A,\B)$ of classes of modules such that $\A^\perp = \B$ and $\A = {^\perp\B}$. Cotorsion pairs are most useful if they are \emph{complete}, i.e.\ for each $M\in\ModR$ there exist short exact sequences
\[
0\longrightarrow B^M\longrightarrow A^M \overset{p}\longrightarrow M \longrightarrow 0
\qquad\textrm{and}\qquad
0\longrightarrow M \overset{i}\longrightarrow B_M \longrightarrow A_M \longrightarrow 0
\]
with $A^M,A_M \in \A$ and $B^M,B_M \in \B$. The map $p$ is called a \emph{special $\A$-precover} of $M$ while $i$ is called a \emph{special $\B$-preenvelope} of $M$. A key observation~\cite[Theorem~6.11(b)]{GT} is that any cotorsion pair generated by a set $\clS$ of modules, i.e.\ such that $\B=\clS^\perp$, is complete. In practice, naturally arising cotorsion pairs are usually proved to have this property.

A cotorsion pair $\mathfrak{C} = (\A,\B)$ is \emph{hereditary} if, $\Ext^n_R(A,B)=0$ for each $A\in\A$, $B\in\B$ and $n\ge 1$. Equivalently, one might require that $\A$ be closed under kernels of epimorphisms, or that $\B$ be closed under cokernels of monomorphisms, \cite[Lemma~5.24]{GT}.}

We say that $\mathcal D\subseteq\ModR$ is a \emph{definable class} if $\mathcal D$ is closed under products, direct limits and pure submodules.
For a class $\mathcal C\subseteq\ModR$, we denote by $\Cogen(\mathcal C)$ the class of modules cogenerated by $\mathcal C$, i.e.\ the class of all submodules of products of modules from $\mathcal C$. Analogously, we denote by $\Cogen_*(\mathcal C)$ the class of all pure submodules of products of modules from $\mathcal C$, and by $\bar{\mathcal C}$ the \emph{definable closure} of the class $\mathcal C$, i.e. the smallest definable class containing $\mathcal C$. Furthermore, we use the notations $\mathcal C^\perp = \bigcap _{C\in\mathcal C}\Ker\Ext_R^1(C,-)$ and ${}^\perp\mathcal C = \bigcap _{C\in\mathcal C}\Ker\Ext_R^1(-,C)$. \out{For instance, $\EC = \FL^\perp$ is the class of all (Enochs) cotorsion modules.} If $\mathcal C = \{C\}$, we write just $\Cogen(C)$, $\Cogen_*(C)$, $\bar C$, $C^\perp$ or ${}^\perp C$, respectively.

Every definable class is closed under pure-epimorphic images by \cite[Theorem~3.4.8]{P2}, and the definable closure $\bar{\mathcal C}$ of a class $\mathcal C$ can be constructed, for instance, by closing $\mathcal C$ under products, then under pure submodules and finally under pure-epimorphic images. Note also that, for any definable class $\mathcal D$, there is by \cite[Corollary 5.3.52]{P2} a pure-injective module $C\in\mathcal D$ such that $\Cogen_*(C) = \mathcal D$. The module $C$ is called an \emph{elementary cogenerator} of the definable class $\mathcal D$.

Further, given a right (left, resp.) $R$-module $M$, $M^c$ stands for the \emph{character module of $M$}, i.e.\ the left (right, resp.) $R$-module $\Hom_{\Z}(M,\Q/\Z)$. Recall that if $\mathcal C$ is closed under products and direct limits and $M\in\mathcal C$, then $M^{cc}\in\mathcal C$ (cf.\ \cite[Lemma~5.3]{S}, this is because $M^{cc}$ is elementarily equivalent to $M$, so it purely embeds into an ultrapower of $M$, and hence is a summand there).

Finally, for a regular uncountable cardinal $\lambda$, we call a directed system $\mathcal M = (M_i,f_{ji}\colon M_i\to M_j \mid i< j\in I)$ of modules \emph{$\lambda$-continuous}, provided that the poset $(I,\leq)$ has got suprema of all chains of length $<\lambda$ and, for any such chain $J\subseteq I$, we have $M_{\sup J} = \varinjlim_{j\in J} M_j$. It is easy to see that $\mathcal M$ is then $\lambda$-directed.
\chge{A well-ordered direct system $(M_\alpha,f_{\beta\alpha}\colon M_\alpha\to M_\beta \mid \alpha<\beta<\sigma)$ is called a \emph{filtration} of a~module $M$ if all the maps in the system are inclusions, $M_0=0$, $M_{\theta} = \varinjlim_{\alpha<\theta} M_\alpha$ for each limit ordinal $\theta<\sigma$ and $M=\varinjlim_{\alpha<\sigma} M_\alpha$.}

\smallskip

The following definition contains notions which are fundamental in this paper.

\begin{defn} Let $R$ be a ring and $M,N\in\ModR$. We say that a homomorphism $f\colon M\to N$ is \emph{$\mathcal C$-injective} if $\Hom_R(f,C)$ is surjective for all $C\in\mathcal C$. Moreover, for an uncountable regular cardinal $\lambda$, we say that a module $M$ is \emph{almost $(\mathcal C,\lambda)$-projective}, if $M$ is the direct limit of a $\lambda$-continuous directed system consisting of $<\lambda$-presented modules from ${}^\perp\mathcal C$. If moreover all the colimit maps are $\mathcal C$-injective, then we call the module $M$ \emph{$(\mathcal C,\lambda)$-projective}. If $\mathcal C = \{C\}$, we write just $C$-injective and (almost) $(C,\lambda)$-projective, respectively.
\end{defn}

\begin{rem} If $M$ is $<\lambda$-presented, then (almost) $(\mathcal C,\lambda)$-projectivity of $M$ amounts to $M\in {}^\perp\mathcal C$.
\end{rem}

Let us start with an easy observation.

\begin{lem} \label{l:cogen} If $f\colon M\to N$ is a $\mathcal C$-injective map and $D\in\Cogen(\mathcal C)$ a module with $\Ext_R^1(\Coker(f),D)= 0$, then $f$ is $D$-injective.
\end{lem}

\begin{proof} By our assumptions, $\Ker(f)\subseteq\Ker(h)$ for any $h\in\Hom_R(M,C)$ where $C\in\mathcal C$. It immediately follows that the same holds for any $h\in\Hom_R(M,D)$ since $D\in\Cogen(\mathcal C)$. From the hypothesis on $\Coker(f)$, we see that the inclusion $\Img(f)\subseteq N$ is $D$-injective, hence $f$ is $D$-injective as well.
\end{proof}

We are interested in when $\Ext^1_R(\varinjlim M_i,C) = 0$ holds true for a direct system of modules $(M_i\mid i \in I)$. The following lemma gives us a tool to handle this situation in a special case.

\begin{lem} \label{lem:left_dir_lim}
Let $R$ be a ring, $C$ be a module, and $(M_i, f_{ji} \mid i<j \in I)$ be a directed system of modules \st $\Ext^1_R(M_i,C) = 0$ for each $i \in I$. Then the following are equivalent:
\begin{enumerate}
\item $\Ext^1_R(\varinjlim M_i,C) = 0$.
\item For each family $(g_{ji}\colon M_i \to C \mid i<j)$ of morphisms \st 
$$
g_{ki} = g_{ji} + g_{kj} f_{ji} \qquad
\textrm{for each $i,j,k \in I$ with $i<j<k$},
$$
there is a family $(g_i\colon M_i \to C \mid i \in I)$ of morphisms \st
$$
g_i = g_{ji} + g_j f_{ji} \qquad
\textrm{for each $i,j \in I$ with $i<j$}.
$$
\end{enumerate}
\end{lem}

\begin{proof}
It is well known that there is the following exact sequence for $\varinjlim M_i$:
$$
\dots \overset{\delta_2}\to
\bigoplus_{i_0<i_1<i_2} M_{i_0 i_1 i_2} \overset{\delta_1}\to
\bigoplus_{i_0<i_1} M_{i_0 i_1} \overset{\delta_0}\to
\bigoplus_{i_0 \in I} M_{i_0} \to
\varinjlim M_i \to 0 $$
where $M_{i_0 i_1 \dots i_n} = M_{i_0}$ for all $i_0<i_1< \dots <i_n$ in $I$ and
\begin{align*}
\big( \delta_0 \restriction M_{ij} \big)(x) &= \big( x,-f_{ji}(x) \big) \in M_i \times M_j \\
\big( \delta_1 \restriction M_{ijk} \big)(x) &= \big( x, -x, -f_{ji}(x) \big) \in M_{ik} \times M_{ij} \times M_{jk}
\end{align*}
If we apply the functor $\Hom_R(-,C)$ to that long exact sequence, we get in general a complex. Since $\Ext^1_R(\bigoplus M_{i_0},C) = 0$, we deduce that $\Ext^1_R(\varinjlim M_i,C) = 0$ \iff this complex is exact at $\Hom_R(\bigoplus_{i_0<i_1} M_{i_0 i_1},C)$. However, we have:
\begin{align*}
\Hom_R(\delta^0,C) \big((g_i)_{i}\big) &= (g_i - g_j f_{ji})_{i<j} \\
\Hom_R(\delta^1,C) \big((g_{ji})_{i<j}\big) &= (g_{ki} - g_{ji} - g_{kj} f_{ji})_{i<j<k}
\end{align*}
Hence, the exactness condition translates precisely to the condition (2) of the statement.
\end{proof}

As a fruitful corollary, we obtain the following generalization of what is referred to as the Eklof Lemma in \cite[Lemma 6.2]{GT}.

\begin{lem} \label{l:eklof} Let $R$ be a ring. Let $\mathcal C\subseteq\ModR$, $\sigma$ be a limit ordinal and $\mathcal M = (M_\alpha,f_{\beta\alpha}\colon M_\alpha\to~ M_\beta \mid \alpha<\beta\leq\sigma)$ be a continuous well-ordered direct system of modules such that $M_{\alpha}\in {}^\perp \mathcal C$ and $f_{\alpha+1,\alpha}$ is $\mathcal C$-injective for all $\alpha<\sigma$. Then $M_\sigma\in {}^\perp\mathcal C$.
\end{lem}

\begin{proof} Pick any $C\in\mathcal C$ and suppose that, as in the condition $(2)$ of Lemma~\ref{lem:left_dir_lim}, we have a collection of maps $(g_{\beta\alpha})_{\alpha<\beta}$ with $g_{\beta\alpha}\colon M_\alpha \to C$ and $g_{\gamma\alpha} = g_{\beta\alpha} + g_{\gamma\beta} f_{\beta\alpha}$ whenever $\alpha<\beta<\gamma<\sigma$. We need to find $g_\alpha\colon M_\alpha \to C$ \st
\[ g_\alpha = g_{\beta\alpha} + g_\beta f_{\beta\alpha} \qquad
\textrm{for each $\alpha<\beta<\sigma$.}
\]
By substracting $g_{\gamma\alpha} = g_{\beta\alpha} + g_{\gamma\beta} f_{\beta\alpha}$, this can be equivalently reformulated to
\[ g_\alpha - g_{\gamma\alpha} = (g_\beta - g_{\gamma\beta}) f_{\beta\alpha} \qquad
\textrm{for each $\alpha<\beta<\gamma<\sigma$.}
\]

We can construct such $g_\alpha$ by transfinite induction on $\alpha<\sigma$. The initial morphism $g_0$ can be chosen arbitrarily. If $\beta=\alpha+1$ is a successor ordinal, we take $g_\beta$ as a lift of $g_\alpha - g_{\beta\alpha}\colon M_\alpha \to C$ over $f_{\beta\alpha}\colon M_\alpha \to M_\beta$, using the $\mathcal C$-injectivity of $f_{\beta\alpha}$. Finally, if $\alpha$ is a limit ordinal, we use the continuity of the direct system and the fact that $(g_{\delta} - g_{\alpha\delta} \mid \delta<\alpha)$ is a cocone of the direct system $(M_\delta \mid \delta<\alpha)$, and take $g_\alpha\colon M_\alpha \to C$ as the colimit map corresponding to this cocone.
\end{proof}

\begin{rem}It follows from the lemma above that, given an uncountable regular cardinal $\kappa$, a $\kappa$-presented module which is $(\mathcal C,\kappa)$-projective belongs to ${}^\perp\mathcal C$. Indeed, if $\mathcal M = (M_i,f_{ji}:M_i\to M_j \mid i< j\in I)$ is any $\kappa$-continuous directed system of $<\kappa$-presented modules whose direct limit is $M$, it has a continuous well-ordered subchain $\mathcal M' = (M_{i_\alpha} \mid \alpha<\kappa)$ with the same direct limit. If $\mathcal M$ witnesses the $(\mathcal C,\kappa)$-projectivity of $M$, so does $\mathcal M'$ and we can use the lemma.

If, on the other hand, $\kappa$ is singular, we have Lemma~\ref{l:singmod} instead. We will use the lemma here, but postpone its fairly technical proof to Sections~\ref{sec:singset} and~\ref{sec:singmod} for the sake of better readability.
\end{rem}

Recall from~\cite{AH} that, for $\mathcal Q\subseteq\RMod$, a module $M$ is \emph{$\mathcal Q$-Mittag-Leffler}, if the canonical morphism $\rho\colon M\otimes_R\prod_{i\in I} Q_i \to \prod_{i\in I}(M\otimes _R Q_i)$ is injective for any subset $\{Q_i \mid i\in I\}$ of $\mathcal Q$.

Further, given a module $M$ and $\mathcal C\subseteq\ModR$, we say that $M$ is \emph{$\mathcal{C}$-stationary} provided that for some (equivalently any) directed system $\mathcal F = (F_i,f_{ji}\colon F_i\to F_j \mid i<j\in I)$ consisting of finitely presented modules with $\varinjlim\mathcal F = M$, the corresponding inverse system $\Hom_R(\mathcal F,C)$ of abelian groups satisfies the Mittag-Leffler condition for each $C\in\mathcal C$. This means that, for any $C\in\mathcal C$ and $i\in I$, there exists $j\in I$, $j\geq i$, such that, for all $g\in\Hom_R(F_j,C)$, we have $gf_{ji}\in\Img(\Hom_R(f_{ki},C))$ for any $i\leq k\in I$. Moreover, if we denote by $f_i: F_i\to M$ the canonical colimit map, we say that $M$ is \emph{strict $\mathcal C$-stationary} if we have even $gf_{ji}\in\Img(\Hom_R(f_i,C))$. These two concepts coincide for $C$ (locally) pure-injective. See \cite[Section~2]{H} for this result and other ones relating the notions of (strict) stationary and Mittag-Leffler module.

\smallskip

One can use the following lemma to present a large $C$-stationary module as the direct limit of a $\lambda$-continuous directed system consisting of small $C$-stationary modules. Compare it with \cite[Theorem 2.6]{HT}. Here we denote, for a set $I$ and a~cardinal number $\lambda$, by $[I]^{<\lambda}$ the set of all subsets of $I$ of cardinality $<\lambda$.

\begin{lem} \label{l:Cext} Let $C$ be a module, and $M$ be a $C$-stationary module. Then for each uncountable regular cardinal $\lambda$, there exists a $\lambda$-continuous directed system $\mathcal L$ consisting of $<\lambda$-presented $C$-stationary modules \st $M=\varinjlim \mathcal L$. Moreover, for every $L$ from the system $\mathcal L$ and a pure-injective module $D\in\bar C$, the canonical colimit map $L\to M$ is $D$-injective.
\end{lem}

\begin{proof} Consider a direct system $\mathcal F = (F_i,f_{ji}\colon F_i\to F_j \mid i< j\in I)$ consisting of finitely presented modules \st $M = \varinjlim _{i\in I}F_i$,
and denote by $f_i\colon F_i\to M$ the canonical colimit maps. We can w.l.o.g.\ assume that $(I,\leq)$ does not have the largest element and, using \cite[Corollary 2.10, Theorem 2.11]{H}, that $C$ is an elementary cogenerator of $\bar C$. Thus each pure-injective $D\in\bar C$ is a direct summand in a product of copies of $C$.

Since $M$ is strict $C$-stationary, we can define a map $\sigma\colon I\to I$ so that for each $i\in~\!\! I$, any cardinal $\kappa$ and each $g\in\Hom_R(F_{\sigma(i)},C^\kappa)$, we have $gf_{\sigma(i)i}\in\Img(\Hom_R(f_i,C^\kappa))$. Further, we fix a map $\delta\colon I^2\to I$ satisfying $i,j\leq\delta(i,j)$. For each $X\in[I]^{<\lambda}$, we construct the set $\tilde X$ as the union of a chain of the sets $X_1=X$, $X_{2n}=X_{2n-1}\cup\delta (X_{2n-1}^2)$, and $X_{2n+1}=X_{2n}\cup\sigma (X_{2n})$ for $1\leq n<\omega$.

Then $\tilde X$ is a directed subposet of $(I,\leq)$ closed under $\sigma$ and $\delta$. Moreover, $|\tilde X|<\lambda$ since $\lambda$ is uncountable. We put $\mathcal L = (\varinjlim _{i\in\tilde X}F_i \mid X\in [I]^{<\lambda})$ where we take the canonical colimit factorization maps as morphisms. It is easy to see that $\mathcal L$ is $\lambda$-continuous and that it consists of $<\lambda$-presented modules.

Pick any $L=\varinjlim _{i\in\tilde X}F_i\in\mathcal L$ and, for all $i\in\tilde X$, denote by $g_i\colon F_i\to L$ the colimit maps. Finally, let $g\colon L\to M$ be the canonical factorization, so that $f_i = g g_i$.

By the construction, $L$ is strict $C$-stationary. Let $D\in\bar C$ be arbitrary pure-injective and let $(-)^*$ denote the functor $\Hom_R(-,D)$. It remains to show that the map $g^*\colon M^* \to L^*$ is surjective. Let us write $g^*\colon M^*\overset{\theta}{\to} (\Img(g))^*\overset{\iota}{\to} L^*$.

First notice that for each $h\colon L\to D$, we have $\Ker(g)\subseteq \Ker(h)$: indeed, if $y\not\in\Ker(h)$, then $hg_i(x)\neq 0$ for some $i\in\tilde X$ where $g_i(x)=y$. However, by the construction, $hg_i\in\Img(\Hom_R(f_i,D))$, and so $f_i(x)=gg_i(x)\neq 0$ as well. Hence $y\not\in\Ker(g)$. It follows that $\iota$ is an isomorphism and it suffices to prove the surjectivity of $\theta$. 

To this end, consider the short exact sequence $0\to \Img(g)\to M\to\Coker(g)\to 0$ as the direct limit of a direct system $\mathcal D$ of short exact sequences $0\to \Img(f_i)\to M\to\Coker(f_i)\to 0$, where $i$ runs through $\tilde X$. Since $D$ is pure-injective, we have 
$$\varprojlim\Ext^1_R(\Coker(f_i),D)\cong\Ext^1_R(\Coker(g),D)$$
\noindent by a classic result of Auslander (see e.g.\ \cite[Lemma 6.28]{GT}). 

Applying the functor $(-)^*$ to the direct system $\mathcal D$, we obtain for each $i\in\tilde X$ the following commutative diagram with exact rows where $\iota_i\colon\Img(f_i) \to \Img(g)$ denotes the inclusion:
$$\begin{CD}
	M^* @>{\theta}>>	(\Img(g))^* @>{\eta}>> \Ext^1_R (\Coker(g),D) @. 	\\
	@|			@V{\iota_i ^*}VV	 @V{\varepsilon_i}VV \\
	M^* @>{\iota_i^*\theta}>>	(\Img(f_i))^* @>{\eta_i}>> \Ext^1_R (\Coker(f_i),D). @. \\
\end{CD}$$
We will prove that $\eta$ is the zero map. By the previous paragraph, this amounts to showing that $\varepsilon_i\eta =0$ for all $i\in\tilde X$. However, $\varepsilon_i\eta=\eta_i \iota_i^*$, and $\Img(\iota_i^*)\subseteq\Img(\iota_i^*\theta) = \Ker(\eta_i)$, by the construction of~$L$. Hence $\eta_i \iota_i^* = 0$ for each $i\in\tilde X$ and $\theta$ is an epimorphism. 
\end{proof}

\chge{
\begin{rem}
Notice that the same proof would apply if we allow the $L$ from the statement of Lemma~\ref{l:Cext} to be the direct limit of a directed subsystem of $\mathcal L$.
\end{rem}
}

\begin{defn} \label{d:inducesys} Let $\mathcal S = (M_i,f_{ji}:M_i\to M_j \mid i<j\in I)$ be a directed system of modules. For each $\eta$ regular uncountable, we denote by $\mathcal S^\eta$ the directed system of modules consisting of direct limits of directed subsystems of $\mathcal S$ of cardinality $<\eta$ and canonical factorization maps between them. So if $g_{lk}: N_k \to N_l$ is a morphism in $\mathcal S^\eta$, then $N_k = \varinjlim \mathcal D_k, N_l = \varinjlim \mathcal D_l$ where $\mathcal D_k, \mathcal D_l$ are directed subsystems of $\mathcal S$ of cardinality $<\eta$ and $\mathcal D_k\subseteq \mathcal D_l$.
\end{defn}

\chge{
\begin{rem}
The latter definition is often used in conjunction with Observation~\ref{o:observ}. The typical use is as follows. If $\clS$ is $\aleph_1$-continuous and $\D\in\clS^\eta$, then $\D$ is often not $\aleph_1$-continuous itself. Observation~\ref{o:observ} allows us to find an $\aleph_1$-continuous direct system $\D'$ such that $\D\subseteq\D'\subseteq\clS$ and $\varinjlim\D=\varinjlim\D'$. Thus, for instance, if $\clS$ witnesses almost $(\C,\aleph_1)$-projectivity, so does $\D'$.
\end{rem}
}

In the lemma below, we use the notion of a filter-closed class from \cite{SS}. Note that a class is filter-closed, for instance, provided that it is closed under direct products and either direct limits of monomorphisms, or pure submodules. On the other hand, any filter-closed class is closed under taking arbitrary direct products and direct sums.

\begin{lem} \label{l:filter-closed} Let $\mathcal B$ be a filter-closed class of modules and $M\in {}^\perp\mathcal B$ an almost $(\mathcal B,\aleph_1)$-projective module. Then $M$ is $(\Cogen_*(\mathcal B),\lambda)$-projective for any regular uncountable cardinal $\lambda$. In particular, $M\in {}^\perp\Cogen_*(\mathcal B)$.
\end{lem}

\begin{proof} Put $\mathcal D = \Cogen_*(\mathcal B)$ and $\mathcal A = {}^\perp\mathcal D$. Note that $\mathcal D$ is necessarily filter-closed. Let $\mathcal T$ be a directed system witnessing that $M$ is almost $(\mathcal B,\aleph_1)$-projective. Let $\kappa$ be the least infinite cardinal such that $M$ is $\kappa$-presented. We proceed by induction on $\kappa$.

For $\kappa = \aleph_0$, we have to show only that $M\in \mathcal A$. However, this follows immediately from \cite[Proposition 4.3]{S}.

Now, let $\kappa$ be uncountable.
By induction on $\aleph_0<\lambda\leq\kappa$, $\lambda$ regular, we prove that there exists a directed system $\mathcal C_\lambda$ consisting of modules from $\mathcal A$ and witnessing that $M$ is $(\mathcal B,\lambda)$-projective. For $\lambda = \aleph_1$, \cite[Proposition 4.3]{S} gives us that $\mathcal T$ witnesses almost $(\mathcal D,\aleph_1)$-projectivity of~$M$, which in fact means none other than that $\mathcal T$ consists of modules from $\mathcal A$. Using \cite[Lemma 2.3]{SS}, we obtain a~directed subsystem $\mathcal C_{\aleph_1}$ of $\mathcal T$ witnessing $(\mathcal B,\aleph_1)$-projectivity of $M$.

Let $\lambda>\aleph_1$. Consider the system $\mathcal T^\lambda$ (see Definition~\ref{d:inducesys}). By Observation~\ref{o:observ}, the modules from $\mathcal T^\lambda$ are almost $(\mathcal B,\aleph_1)$-projective. As before, \cite[Lemma 2.3]{SS} provides us with a $\lambda$-continuous directed subsystem $\mathcal C$ of $\mathcal T^\lambda$ \st $\varinjlim\mathcal C = M$ and all the colimit maps from modules in $\mathcal C$ to $M$ are $\mathcal B$-injective. We are going to show that $\mathcal C$ contains a $\lambda$-continuous directed subsystem consisting of modules from $\mathcal A$. In fact, it is enough to show that modules from $\mathcal A$ occur cofinally in the directed system~$\mathcal C$. Indeed, suppose we are given a well-ordered subsystem $\mathcal S$ of $\mathcal C$ of cardinality $<\lambda$ consisting of modules from~$\mathcal A$. Since all the connecting maps in $\mathcal S$ are $\mathcal B$-injective, it follows from Lemma~\ref{l:eklof} that $\varinjlim\mathcal S$ belongs to ${}^\perp\mathcal B$, and hence to $\mathcal A$ by the inductive hypothesis for $\kappa$ (we have $\lambda\leq\kappa$ and $\varinjlim \mathcal S$ is $<\lambda$-presented). We distinguish the following three cases.

\smallskip

Case 1: $\lambda = \eta^+$ for $\eta$ regular. Consider the system $\mathcal C_\eta^\lambda$ which comes from the application of the construction in Definition~\ref{d:inducesys} to the system $\mathcal C_\eta$, which witnesses the $(\mathcal B,\eta)$-projectivity of $M$. Notice in particular that each direct subsystem $\mathcal S\subseteq \mathcal C_\eta$ of cardinality $\leq\eta$ either has a supremum in $\mathcal C_\eta$ (if the cardinality of $\mathcal S$ is $<\eta$) or has a cofinal well-ordered subsystem $\mathcal S'$ which is indexed by $\eta$. In the latter case, we can replace $\mathcal S$ by $\mathcal S'$ and assume w.l.o.g.\ that $\mathcal S$ is continuous (since $\mathcal C_\eta$ is $\eta$-continuous). As above, it follows from Lemma~\ref{l:eklof} and the inductive hypothesis that all elements of $\mathcal C_\eta^\lambda$ belong to $\mathcal A$.

Finally, we can, by Construction~\ref{conmerge}, intersect the system $\mathcal C$ above with the system $\mathcal C_\eta^\lambda$. The resulting $\lambda$-continuous common subsystem $\mathcal C_\lambda$ has all the properties which we require.

\smallskip

Case 2: $\lambda = \eta^+$ for $\eta$ singular. Set $\mathcal M_\lambda = \mathcal C$ and, for all regular uncountable $\theta<\eta$, let us denote by $\mathcal M_\theta$ the system $\mathcal C_\theta^\lambda$. The set $\{\mathcal M_\theta \mid \aleph_0<\theta = \cf(\theta)\leq\lambda\}$ has cardinality $\leq\eta$ which allows us to use Construction~\ref{conmerge} to intersect all the systems into one which we denote by $\mathcal C_\lambda$. Then objects of $\mathcal C_\lambda$ are in ${}^\perp \mathcal B$ by Lemma~\ref{l:singmod} (see also the remark below Lemma~\ref{l:eklof}) since each object in $\mathcal C_\lambda$ is an $\eta$-presented module which is almost $(\mathcal D,\theta)$-projective for all regular uncountable $\theta<\eta$. Thus, the objects in $\mathcal C_\lambda$ are also in $\mathcal A$ by the inductive hypothesis and the maps in $\mathcal C_\lambda$ are $\mathcal B$-injective since such are the maps in $\mathcal C$.

\smallskip

Case 3: $\lambda$ is a limit regular cardinal. We can denote $\mathcal C = (M_i,f_{ji}\colon M_i\to M_j\mid i< j\in I)$, and for each $i\in I$, let $f_i\colon M_i \to M$ be the canonical colimit map. We construct an increasing countable subposet of $(I,\leq)$ as follows: 

Pick any $i_0\in I$. Assume that $i_n$ is defined and that $M_{i_n}$ is $\eta$-presented for $\aleph_0<\eta<\lambda$. We pick a module $N_n$ from the system $\mathcal C_{\eta^+}$ and a factorization $g_n\colon M_{i_n}\to N_n$ of $f_{i_n}$ through the $\mathcal B$-injective canonical colimit map $h_n\colon N_n\to M$. Subsequently, we pick $i_{n+1}\in I, i_n<i_{n+1}$, and $g_n^\prime\colon N_n\to M_{i_{n+1}}$ in such a way that $h_n = f_{i_{n+1}}g_n^\prime$ and $f_{i_{n+1}i_n} = g_n^\prime g_n$.

The module $M_{i_\omega}=\varinjlim _{n<\omega}M_{i_n}$ is a $<\lambda$-presented member of the system $\mathcal C$. On the other hand $M_{i_\omega}$ is the direct limit of the directed system

$$N_0\overset{g_1g_0^\prime}{\longrightarrow} N_1\overset{g_2g_1^\prime}{\longrightarrow} N_2 \overset{g_3g_2^\prime}{\longrightarrow}\dotsb$$
consisting of modules from $\mathcal A$ and $\mathcal B$-injective connecting maps. It follows that $M_{i_\omega}$ belongs to ${}^\perp\mathcal B$ by Lemma~\ref{l:eklof}, hence $M_{i_\omega}\in\mathcal A$ by the inductive hypothesis.

Since $i_0\in I$ was arbitrary, we have shown that modules from $\mathcal A$ occur in the system $\mathcal C$ cofinally, whence there is a subsystem $\mathcal C_\lambda$ of $\mathcal C$ consisting of modules from $\mathcal A$ and witnessing $(\mathcal B,\lambda)$-projectivity of $M$.

\smallskip

We have finished the construction of the systems $\mathcal C_\lambda$. Notice that each module in $\mathcal C_\lambda$ is almost $(\mathcal B,\aleph_1)$-projective since $\mathcal C_\lambda\subseteq \mathcal T^\lambda$. It also belongs to ${}^\perp\mathcal B$ and so it is even $(\mathcal B,\aleph_1)$-projective by \cite[Lemma 2.3]{SS}. Further, since any morphism $f\colon M_i \to M_j$ in $\mathcal C_\lambda$ is $\mathcal B$-injective and $M_j\in{}^\perp\mathcal B$, the cokernel of $f$ belongs to ${}^\perp\mathcal B$. Moreover, if we apply Construction~\ref{concoker} to $f\colon M_i \to M_j$ and the direct systems $\mathcal M$ and $\mathcal N$ which witness that $M_i$ and $M_j$ are $(\mathcal B,\aleph_1)$-projective, respectively, one directly checks that the maps $u_k$ from the conclusion are $\mathcal B$-injective. Hence $\Coker(f)$ is an almost $(\mathcal B,\aleph_1)$-projective $<\lambda$-presented module in ${}^\perp\mathcal B$ and, by the induction hypothesis, $\Coker(f)\in\mathcal A = {}^\perp \mathcal D$. By Lemma~\ref{l:cogen}, the morphisms in $\mathcal C_\lambda$ are even $\mathcal D$-injective.

Now, if $\kappa$ is regular, $\mathcal C_\kappa$ contains a well-ordered directed subsystem $\mathcal S$ \st $\varinjlim\mathcal S = M$. Using Lemma~\ref{l:eklof}, we deduce that $M\in\mathcal A$. If, on the other hand, $\kappa$ is singular, we use that $\mathcal D$ is filter-closed and apply Lemma~\ref{l:singmod} to show that $M\in\mathcal A$.
Finally, we can use \cite[Lemma 2.3]{SS} once more to pick, for each $\lambda\leq\kappa$ regular uncountable, from the system $\mathcal C_\lambda$ a subsystem witnessing that $M$ is $(\mathcal D,\lambda)$-projective.
\end{proof}

Our first theorem constitutes a partial converse of the remark after Lemma~\ref{l:eklof}.

\begin{thm}\label{t:main1} Let $\mathcal B\subseteq\ModR$ be a class closed under direct limits and products.
Then each module from ${}^\perp\mathcal B$ is $(\bar{\mathcal B},\lambda)$-projective for any $\lambda$ regular uncountable. Consequently, ${}^\perp\mathcal B = {}^\perp\bar{\mathcal B}$.
\end{thm}
\begin{proof} By the assumption on $\mathcal B$ (see \cite[Lemma 5.3]{S}), an elementary cogenerator $C$ of $\bar{\mathcal B}$ is contained in $\mathcal B$. In particular, $\bar{\mathcal B} = \Cogen_*(\mathcal B)$. Our result will follow from Lemma~\ref{l:filter-closed} once we show that each $M\in {}^\perp\mathcal B$ is almost $(\mathcal B,\aleph_1)$-projective.


\smallskip

To this end, note that $M$ is strict $C$-stationary by \cite[Lemma~4.2]{S}.
If we fix a~short exact sequence $$0\longrightarrow N\overset{f}{\longrightarrow} P\overset{}{\longrightarrow}M\longrightarrow 0$$ with $P$ projective, then, since $(C^{(I)})^{cc}\in\mathcal B$ for all sets $I$ and since $P$ is strict $C$-stationary because it is projective, $N$ is strict $C$-stationary as well by \cite[Lemma~4.4]{S}. 

Let $\mathcal L$ be an $\aleph_1$-continuous directed system provided by Lemma~\ref{l:Cext} for $N$, and let $\mathcal S$ be the directed system consisting of all countably generated direct summands of $P$ and inclusions. By Construction~\ref{concoker}, we obtain an $\aleph_1$-continuous directed system $\mathcal K$, consisting of the cokernels of the morphisms $u_k$ from the construction, such that $\varinjlim\mathcal K = M$. Notice that modules from $\mathcal K$ are in ${}^\perp C$. Indeed, $f$ is $C$-injective since $C\in\mathcal B$, and each colimit map from the directed system $\mathcal L$ is $C$-injective by Lemma~\ref{l:Cext}, so each $u_k$ is $C$-injective and the claim follows.
An application of Lemma~\ref{l:Cext} to $M$ provides us, on the other hand, with an $\aleph_1$-continuous directed system $\mathcal K^\prime$ whose all objects are $C$-stationary. If we intersect the two systems using Construction~\ref{conmerge}, we obtain an $\aleph_1$-continuous directed system consisting of countably presented modules which are $C$-stationary and in ${}^\perp C$, hence also in ${}^\perp\mathcal B$ by \cite[Proposition~4.3]{S}. This direct system witnesses that $M$ is almost $(\mathcal B,\aleph_1)$-projective.
\end{proof}


We can also deduce the following crucial result which generalizes~\cite[Theorem~8.17]{GT}. For the first time, it appeared in an unpublished manuscript \cite{St}.

\begin{thm} \label{t:regular_filt}
Let $\theta\leq\kappa$ be uncountable cardinals, $\theta$ regular. Let $C, M$ be modules and $(M_\alpha, f_{\beta\alpha}\mid \alpha <\beta\leq\theta)$ be a $\theta$-continuous directed system \st $M=M_\theta$ and all $M_\alpha$ are $<\theta$-generated modules for $\alpha<\theta$. Suppose that $\Ext^1_R(M_\alpha,C^{(\kappa)}) = 0$ for all $\alpha<\theta$. Then the following conditions are equivalent:
\begin{enumerate}
\item $\Ext^1_R(M,C^{(\kappa)}) = 0$.
\item There is a closed unbounded subset $X \subseteq \theta$ \st $f_{\beta\alpha}$ is $C^{(\kappa)}$-injective for all $\alpha, \beta \in X$, $\alpha<\beta$.
\end{enumerate}
\end{thm}

\begin{rem} For the implication $(2)\Longrightarrow(1)$, we do not need the assumption that the modules $M_\alpha$ are $<\theta$-generated for $\alpha<\theta$.
\end{rem}

\begin{proof}
(2) $\implies$ (1). We can w.l.o.g.\ assume that $f_{\beta\alpha}$ is $C^{(\kappa)}$-injective for each $\alpha<\beta<\theta$ and use Lemma~\ref{l:eklof}.

(1) $\implies$ (2). Put $D = C^{(\kappa)}$. Possibly restricting ourselves just to indices in some closed unbounded subset of $\theta$, we can always assume that whenever $f_{\beta\alpha}$ is not $D$-injective for some $\alpha<\beta$, then already $f_{\alpha+1,\alpha}$ was not $D$-injective.

Suppose now for contradiction that the set
$$ E = \{ \alpha < \theta \mid f_{\alpha+1,\alpha} \textrm{ is not }D\textrm{-injective} \} $$
is stationary in $\theta$; that is, it intersects every closed unbounded subset of $\theta$. Fix $h_\alpha \in \Hom_R(M_\alpha,D)$ \st $h_\alpha$ does not factorize through $f_{\alpha+1,\alpha}$ for each $\alpha \in E$. Put $h_\alpha=0$ for $\alpha \in \theta \setminus E$.

We will inductively construct homomorphisms $g_{\beta\alpha}: M_\alpha \to D^{(\beta)}$ for all $\alpha<\beta\leq\theta$ \st
\begin{enumerate}
\item[(a)] $g_{\alpha+1,\alpha}$ is the composition of $h_\alpha$ with the $\alpha$th canonical inclusion $D \to D^{(\alpha+1)}$ (in particular, $\Img(g_{\alpha+1,\alpha})\cap D^{(\alpha)} = \{0\}$), and
\item[(b)] $g_{\gamma\alpha} = g_{\beta\alpha} + g_{\gamma\beta} f_{\beta\alpha}$ for all $\alpha<\beta<\gamma\leq\theta$.
\end{enumerate}
For $\beta = 1$, we just put $g_{10} = h_0$. Suppose we have constructed $g_{\beta\alpha}$ for all $\alpha<\beta<\gamma$ for some $\gamma\leq\theta$. If $\gamma=\delta+1$ for some $\delta$, put $g_{\gamma\delta} = h_\delta$ and $g_{\gamma\alpha} = g_{\delta\alpha} + h_\delta f_{\delta\alpha}$ for each $\alpha<\delta$. If $\gamma$ is a limit ordinal, define $g_{\gamma\alpha}$ as the maps $g_\alpha$ given by Lemma~\ref{lem:left_dir_lim}, condition (2). It is straightforward to check that the maps defined in this way satisfy the required conditions.

Put
$$ X = \{\lambda<\theta \mid \Img g_{\theta\alpha} \subseteq D^{(\lambda)} \textrm{ for each } \alpha<\lambda \} $$
It is easy to check that $X$ is closed unbounded in $\theta$. Hence, the set $X'$ consisting of the limit ordinals in $X$ is closed unbounded too, and there is some $\lambda \in X' \cap E$.

Denote by $\pi$ the $\lambda$th canonical projection $D^{(\theta)} \to D$. First we show that $\pi g_{\theta\lambda} =~0$. Choose an arbitrary $x \in M_\lambda$. Since the $M_\lambda = \varinjlim_{\mu<\lambda} M_\mu$ by continuity of the direct system, there is $\alpha<\lambda$ and $y \in M_\alpha$ \st $x = f_{\lambda\alpha}(y)$. We have the equality:
$$ \pi g_{\theta\alpha}(y) = \pi g_{\lambda\alpha}(y) + \pi g_{\theta\lambda} f_{\lambda\alpha}(y) $$
But $\pi g_{\theta\alpha}(y) = 0$ since $\lambda \in X$ and $\pi g_{\lambda\alpha}(y) = 0$ by definition of $g_{\lambda\alpha}$. Hence $0 = \pi g_{\theta\lambda} f_{\lambda\alpha}(y) = \pi g_{\theta\lambda}(x)$. The claim follows since $x \in M_\lambda$ was arbitrary.

On the other hand, we know that $g_{\theta\lambda} = g_{\lambda+1,\lambda} + g_{\theta,\lambda+1} f_{\lambda+1,\lambda}$. Composing this with $\pi$, we get:
$$
0 = \pi g_{\theta\lambda} =
\pi g_{\lambda+1,\lambda} + \pi g_{\theta,\lambda+1} f_{\lambda+1,\lambda} =
h_\lambda + \pi g_{\theta,\lambda+1} f_{\lambda+1,\lambda}
$$
But this implies that $h_\lambda$ factorizes through $f_{\lambda+1,\lambda}$, a contradiction to the choice of $h_\lambda$ for $\lambda \in E$.

Hence, $E$ is not stationary. Therefore, we can choose a closed unbounded subset $X \subseteq \theta$ \st $X \cap E = \varnothing$ and (2) follows.
\end{proof}

\section{\texorpdfstring{$\Sigma$-cotorsion modules and $C$-stationarity}{Sigma-cotorsion modules and C-stationarity}}
\label{sec:sigmacot}

\chge{Let $\FL$ denote the class of all flat $R$-modules and $\EC = \FL^\perp$. The modules in $\EC$ are called \emph{(Enochs) cotorsion modules}. They generalize pure-injective modules and have been studied from that perspective in the series of papers~\cite{G-AH_lc-rings,G-AH_survey,G-AH_indec}, especially regarding their direct sum decomposition properties.
Among cotorsion modules, one may specialize to \emph{$\Sigma$-cotorsion modules}, i.e.\ those whose every sum of copies is cotorsion.
This is an intriguing class of modules at the boundary of homological algebra, model theory and set theory. These modules are far more complicated than $\Sigma$-pure-injective ones and have been studied in \cite{BS_sigma-cot,G-AH_sigma-rings,G-AH_model-th}.

In this section, we prove that being $\Sigma$-cotorsion is a property of the first-order theory of a module rather than the individual module alone, and give an analysis of the resulting theory, extending jointly the descriptions in~\cite[Theorem 12]{G-AH_model-th} (for countable rings) and in~\cite[Theorem 3.8]{BS_sigma-cot} (for non-discrete valuation domains) to general rings.

We start with the lemma below, which} has two alternatives, the `almost' and the `full' one, as indicated by brackets. In its proof, we use the notation from Definition~\ref{d:inducesys}.

\begin{lem} \label{l:dirsysmerge} Let $\nu<\mu$ be infinite cardinals and $F$ an $[$almost$]$ $(\mathcal C,\lambda)$-projective module for each regular $\lambda>\nu$ with $\mu^+\geq\lambda$. Then we can choose the system $\mathcal S_{\mu^+}$ witnessing $[$almost$]$ $(\mathcal C,\mu^+)$-projectivity of $F$ in such a way that each module from $\mathcal S_{\mu^+}$ be $[$almost$]$ $(\mathcal C,\eta)$-projective for any regular $\eta>\nu$.
\end{lem}

\begin{proof} By our assumption, there are directed systems $\mathcal T_\lambda$ witnessing that $F$ is [almost] $(\mathcal C,\lambda)$-projective. The modules in $\mathcal T_{\mu^+}$ are $(\mathcal C,\eta)$-projective for any regular $\eta>\mu$ since they belong to ${}^\perp\mathcal C$. We obtain $\mathcal S_{\mu^+}$ by \chge{intersecting} $\mathcal T_{\mu^+}$ with all the systems $\mathcal T_\eta^{\mu^+}$ where $\nu<\eta=\cf(\eta)\leq\mu$ using Construction~\ref{conmerge}; note that, in this case, we have $|\{\eta \mid \nu<\eta=\cf(\eta)\leq\mu\}|<\mu^+$. The rest follows from Observation~\ref{o:observ} \chge{(see also the Remark after Definition~\ref{d:inducesys})}.
\end{proof}


\begin{prop}\label{p:main2} Let $\kappa$ be a nonzero cardinal, $F,C\in\ModR$ with $F$ $\kappa$-presented. Assume that $F$ is almost $(C^{(\kappa)},\lambda)$-projective for each $\lambda$ regular uncountable. Then $F$ is $(\Cogen_*(C),\lambda)$-projective for each regular $\lambda>\aleph_0$, and so $F\in {}^\perp\Cogen_*(C)$.
\end{prop}

\begin{proof} Put $\mathcal C = \{C^{(\kappa)}\}$ and $\mathcal B = \Cogen_*(C)$. We work by induction on $\kappa$. If $\kappa$ is finite, our assumption says that $\Ext^1_R(F,C) = 0$ and $F$ is finitely presented. Consequently, $F\in {}^\perp\mathcal B$.

If $\kappa = \aleph_0$, we have $\Ext^1_R(F,C^{(\aleph_0)})= 0$ and our result follows from \cite[Proposition 2.7]{SS}. Note that for $F$ countably presented, the $(\mathcal B,\lambda)$-projectivity amounts to $F\in {}^\perp\mathcal B$.

Let $\kappa>\aleph_0$ and, for each regular uncountable $\lambda$, let $\mathcal S_\lambda$ denote a directed system of modules witnessing that $F$ is almost $(\mathcal C,\lambda)$-projective. Notice that $\mathcal S_{\aleph_1}$ even witnesses that $F$ is almost $(\mathcal B,\aleph_1)$-projective by \cite[Proposition 2.7]{SS}. Furthermore, using Lemma~\ref{l:dirsysmerge} with $\nu = \aleph_0$, we can w.l.o.g.\ assume that, if $\lambda$ is a~successor cardinal, all the modules in $\mathcal S_\lambda$ are almost $(\mathcal C,\eta)$-projective for all regular uncountable~$\eta$.
Hence they are $(\mathcal B,\eta)$-projective for all regular uncountable $\eta$ by the inductive hypothesis, provided they are $<\kappa$-presented. In particular, $\mathcal S_\lambda$ witnesses almost $(\mathcal B,\lambda)$-projectivity of $F$ for all infinite successor cardinals $\lambda\leq\kappa$.

To conclude our proof, we \chge{will} show that $F\in {}^\perp\mathcal B$ and use Lemma~\ref{l:filter-closed}. \chge{We remind the reader that $F\in{^\perp\C}$ since we assume that $F$ is almost $(C^{(\kappa)},\kappa^+)$-projective. We discuss several cases depending on $\kappa$:

\smallskip

Case 1:	If $\kappa$ is singular, we get $F\in {}^\perp\mathcal B$ by Lemma~\ref{l:singmod}.

\smallskip

Case 2: Suppose that $\kappa=\mu^+$ is successor cardinal.
Then the system $\mathcal S_\kappa$ can be taken w.l.o.g.\ well-ordered, so $\mathcal S_\kappa = (F_\alpha,f_{\beta\alpha}:F_\alpha\to F_\beta\mid \alpha<\beta<\kappa)$ where each $F_\alpha$ is $<\kappa$-presented. 
As was explained above, we can assume that $\mathcal S_\kappa$ consists of modules which are $(\mathcal B,\lambda)$-projective for all regular uncountable cardinals $\lambda$ (in particular $F_\alpha\in{^\perp\B}$).
Finally, by applying Theorem~\ref{t:regular_filt}, we can also assume that $\mathcal S_\kappa$ witnesses $(\mathcal C,\kappa)$-projectivity.

We fix any $\lambda>\aleph_0$ regular and, for each $\alpha<\kappa$, let $\mathcal T_\alpha$ denote a system witnessing that $F_\alpha$ is $(\mathcal B,\lambda)$-projective. If we put $\mathcal M = \mathcal T_\alpha$ and $\mathcal N = \mathcal T_{\alpha +1}$, Construction~\ref{concoker} will provide us with the system $\mathcal K$ which is easily seen to witness almost $(\mathcal C,\lambda)$-projectivity of $\Coker(f_{\alpha+1,\alpha})$: just use the properties of $\mathcal M$ and $\mathcal N$ together with the $\mathcal C$-injectivity of $f_{\alpha+1,\alpha}$. Since $\lambda$ was arbitrary, we can use the inductive hypothesis to deduce that $\Coker(f_{\alpha+1,\alpha})\in {}^\perp\mathcal B$. Subsequently, the map $f_{\alpha+1,\alpha}$ is $\mathcal B$-injective by Lemma~\ref{l:cogen} for every $\alpha<\kappa$, and Lemma~\ref{l:eklof} yields $F\in {}^\perp\mathcal B$.

\smallskip

Case 3: Let $\kappa$ be a weakly inaccessible (i.e.\ uncountable, regular and limit) cardinal.
As in the previous case, $\mathcal S_\kappa$ can be taken well-ordered, of the form $\mathcal S_\kappa = (F_\alpha,f_{\beta\alpha}:F_\alpha\to F_\beta\mid \alpha<\beta<\kappa)$ where each $F_\alpha$ is $<\kappa$-presented.
Possibly by intersecting $\mathcal S_\kappa$ with $\mathcal S_{\aleph_1}^\kappa$, we can also w.l.o.g.\ assume that $\mathcal S_\kappa$ consists of almost $(\mathcal B,\aleph_1)$-projective modules.

We next show that $\mathcal S_\kappa$ contains a cofinal subsystem consisting of modules which are $(\mathcal B,\lambda)$-projective for each regular uncountable~$\lambda$. To this end, suppose that $M_0$ is an arbitrary module from $\mathcal S_\kappa$ (i.e.\ $M_0 = F_\alpha$ for some $\alpha<\kappa$). Then $M_0$ is $\kappa_0$-presented for an infinite cardinal $\kappa_0<\kappa$. Let $\mathcal R_0$ be the directed system obtained by intersecting $\mathcal S_\kappa$ with the systems $\mathcal S_\lambda^\kappa$ where $\lambda$ runs through the uncountable regular cardinals $\leq\kappa_0^+$ (see Construction~\ref{conmerge}). We continue recursively: once $M_n$, $\kappa_n$ and $\mathcal R_n$ are defined, we find a $\kappa_{n+1}$-presented module $M_{n+1}\in\mathcal R_n$ above $M_n$ where $\kappa_n< \kappa_{n+1}< \kappa$; finally, we let $\mathcal R_{n+1}$ be the intersection of $\mathcal R_n$ with the systems $\mathcal S_\lambda^\kappa$ where $\kappa_n<\lambda = \cf(\lambda)\leq\kappa_{n+1}^+$. Then $M = \varinjlim_{n<\omega}M_n$ belongs to $\mathcal R_n$ for each $n<\omega$, hence $M$ is almost $(\mathcal C,\lambda)$-projective for each $\lambda$ regular uncountable. It follows from the inductive hypothesis that $M$ is $(\mathcal B,\lambda)$-projective for each $\lambda>\aleph_0$ regular.

To summarize so far, we can w.l.o.g.\ assume that $F_\alpha$ is $(\mathcal B,\lambda)$-projective (for each $\lambda>\aleph_0$ regular) whenever $\alpha$ is a non-limit ordinal. Put $F_\kappa = F$ and suppose, for the sake of contradiction, that there exists a limit ordinal $\delta\leq\kappa$ \st $F_\delta$ is not $(\mathcal B,\eta)$-projective for some regular uncountable $\eta$. Take the least such $\delta$.

Similarly to Case 2, we fix a regular cardinal $\lambda>\aleph_0$ and, for each $\alpha<\delta$, denote by $\mathcal T_\alpha$ a system witnessing that $F_\alpha$ is $(\mathcal B,\lambda)$-projective. If we put $\mathcal M = \mathcal T_\alpha$ and $\mathcal N = \mathcal T_{\alpha +1}$, Construction~\ref{concoker} will provide us with the system $\mathcal K$ which witnesses almost $(\mathcal C,\lambda)$-projectivity of $\Coker(f_{\alpha+1,\alpha})$. As before, we can use the inductive hypothesis to deduce that $\Coker(f_{\alpha+1,\alpha})\in {}^\perp\mathcal B$, so that the map $f_{\alpha+1,\alpha}$ is $\mathcal B$-injective by Lemma~\ref{l:cogen} for every $\alpha<\delta$, and Lemma~\ref{l:eklof} yields $F_\delta\in {}^\perp\mathcal B$. Lemma~\ref{l:filter-closed} then gives us the desired contradiction.
}
\end{proof}

\begin{rem} Going through the proof, we can see that, for $\kappa$ singular, the same conclusion holds assuming only that $F$ is almost $(C^{(\kappa)},\lambda)$-projective for all regular uncountable $\lambda<\kappa$.
\end{rem}

\chge{Now we can prove the model-theoretic nature of the $\Sigma$-cotorsion property.

\begin{thm} \label{t:sigmacot} If $C$ is a $\Sigma$-cotorsion module, then every module from the definable closure of $\{C\}$ is $(\Sigma$-$)$cotorsion.
\end{thm}
}

\begin{proof} Let $F$ be an arbitrary flat module. Then $F$ is a $\lambda$-continuous direct limit of $<\lambda$-presented flat modules, hence it is almost $(C^{(\kappa)},\lambda)$-projective for any $\kappa$ and regular $\lambda>\aleph_0$. By Proposition~\ref{p:main2}, we have $F\in {}^\perp\Cogen_*(C)$.
Since the class $\FL$ of all flat modules is resolving, we get even $\mathcal{FL}\subseteq {}^\perp\bar C$ (every $M\in\bar C$ is a~pure-epimorphic image of a module from $\Cogen_*(C)$).
\end{proof}

\begin{cor} \label{c:pure-split} Let $\kappa = |R|+\aleph_0$, $F$ be a flat $\Sigma$-cotorsion module and $\mu$ a cardinal. Then any pure submodule of $F^{(\mu)}$ splits. In particular, $F$ is a direct sum of $\kappa$-presented modules with local endomorphism rings.
\end{cor}

\begin{proof} Let $P$ be a pure submodule in $F^{(\mu)}$. Using Theorem~\ref{t:sigmacot}, we get that $P$ is cotorsion. Consequently $P$ splits in $F^{(\mu)}$ since $F^{(\mu)}/P$ is flat. The rest follows from \cite[Theorem 1.1]{AS}.
\end{proof}

The point in the following rather technical recursive definition is that $\mathfrak S_C(M)$ holds \iff a certain set of first-order sentences, determined by the module $M$, holds in $C$. We can look at $\mathfrak S_C$ as a sort of `hereditary' $C$-stationarity.

\begin{defn} \label{d:S_C} For $M\in\ModR$, we denote by $\pres(M)$ the least cardinal $\theta$ such that $M$ is $\theta$-presented.
For a module $C$, we recursively define a property $\mathfrak S_C$ of a~module $M$ by stating that $\mathfrak S_C(M)$ holds \iff
\smallskip
\begin{enumerate}
	\item $M$ is $C$-stationary and $\pres(M)\leq\aleph_0$, or
	\item $\theta = \pres(M)$ is uncountable regular and there exists a $\theta$-continuous well-ordered directed system $\mathcal M = (M_\alpha,f_{\beta\alpha}:M_\alpha\to M_\beta\mid \alpha<\beta<\theta)$ \st
	\begin{enumerate}
	\item $M_0= 0$ and $\pres(M_\alpha)<\theta$ for each $\alpha<\theta$,
	\item $\varinjlim\mathcal M = M$,
	\item $\mathfrak S_C(\Coker(f_{\beta\alpha}))$ holds for each $\alpha<\beta<\theta$, or
\end{enumerate} 
  \item $\theta = \pres(M)$ is singular and, for each \chge{infinite successor cardinal} $\lambda<\theta$, $M$ is the direct limit of a $\lambda$-continuous directed system consisting of $<\lambda$-presented modules satisfying $\mathfrak S_C$.
\end{enumerate}
\end{defn}

\begin{rem} Notice that $\mathfrak S_C$ implies $\mathfrak S_J$ whenever $J$ lies in the definable closure of $C$ since $C$-stationarity yields $J$-stationarity in this case. Moreover, it follows by \cite[Proposition 2.2]{HT} that $\mathfrak S_C(M)$ implies that $M$ is $C$-stationary for any $M\in\ModR$.
\end{rem}

We have the following interesting characterization of pure-projectivity.

\begin{prop} \label{p:pproj} Let $C$ be an elementary cogenerator of $\ModR$. Then $\mathfrak S_C(M)$ holds \iff $M$ is pure-projective.
\end{prop}

\begin{proof} First, notice that $C$-stationary means Mittag-Leffler. The if part easily follows from the fact that every pure-projective module decomposes as a direct sum of countably presented pure-projective (equivalently countably presented Mittag-Leffler) modules \chge{by \cite[Theorem 26.1]{AF}}.

The only-if part goes by induction on $\kappa = \pres(M)$. For $\kappa\leq\aleph_0$, $C$-stationarity of $M$ amounts to $M$ being pure-projective.

If $\kappa$ is regular uncountable, we consider the system $\mathcal M$ from the definition of~$\mathfrak S_C$. By the remark above, $M$ is $C$-stationary. From Lemma~\ref{l:Cext}, we get a~$\kappa$-continuous directed system $\mathcal L$ consisting of $<\kappa$-presented pure submodules of $M$ and inclusions \st $\varinjlim\mathcal L = M$. \chge{Intersecting} $\mathcal M$ and $\mathcal L$ together by Construction~\ref{conmerge}, we obtain a~filtration $(M_\alpha \mid \alpha<\kappa)$ consisting of $<\kappa$-presented pure submodules of $M$ such that $\mathfrak S_C(M_{\alpha+1}/M_\alpha)$ for each $\alpha<\kappa$. However, by the inductive hypothesis, $M_{\alpha+1}/M_\alpha$ is pure-projective, whence the filtration splits and $M\cong \bigoplus _{\alpha<\kappa} M_{\alpha+1}/M_\alpha$.

If $\kappa$ is singular, we use (again) the definition of $\mathfrak S_C$ together with Lemma~\ref{l:Cext} to obtain, for each $\lambda<\kappa$ infinite successor cardinal, a $\lambda$-continous directed system $\mathcal N$ consisting of $<\lambda$-presented pure submodules of $M$ satisfying $\mathfrak S_C$. Using the inductive hypothesis, we see that the system $\mathcal N$, in fact, consists of (direct sums of countably presented) pure-projective modules. From the classical Shelah's singular compactness theorem \cite[\S 2 II, Theorems 3.1 and 4.1]{E}, we conclude that $M$ is as well a direct sum of countably presented pure-projective modules.
\end{proof}

In the statement of the proposition below, $PE(C)$ denotes the pure-injective envelope of $C$.

\begin{prop} \label{p:stationary} Let $\kappa$ be an infinite cardinal and $M,C\in\ModR$ where $M$ is $\kappa$-presented. Consider the following conditions:
\begin{enumerate}
	\item $M$ is $(\Cogen_*(C),\lambda)$-projective for all regular $\lambda>\aleph_0$;
	\item $M$ is almost $(C^{(\kappa)},\lambda)$-projective for all regular $\lambda>\aleph_0$;
	\item $\mathfrak S_C(M)$ holds.
\end{enumerate}
Then $(1)\Longleftrightarrow (2)\Longrightarrow (3)$. If $M$ is almost $(PE(C),\aleph_1)$-projective, then $(3)\Longrightarrow (1)$.
\end{prop}

\begin{proof} 
The implication $(1)\Longrightarrow (2)$ is trivial. The converse one is exactly Proposition~\ref{p:main2}.

$(1)\Longrightarrow (3)$. By induction on $\kappa$. If $\kappa = \aleph_0$, it follows by \cite[Lemma 4.2]{S}. Let $\kappa$ be uncountable. Using Lemma~\ref{l:dirsysmerge} and the inductive hypothesis, we can assume that, for all \chge{uncountable cardinal successors} $\lambda\leq\kappa$, there is a system $\mathcal S_\lambda$ consisting of modules satisfying $\mathfrak S_C$ such that $\mathcal S_\lambda$ witnesses \out{almost} $(\Cogen_*(C),\lambda)$-projectivity of $M$. If $\kappa$ is singular, we \chge{obtain $\mathfrak S_C(M)$ by the very definition}.

\out{Since we have $M\in {}^\perp\Cogen_*(C)$ by $(1)$, we can assume that $\mathcal S_\kappa$ actually witnesses that $M$ is $(\Cogen_*(C),\lambda)$-projective; just use {\cite[Lemma 2.3]{SS}}.}

\chge{Assume therefore that $\kappa$ is uncountable regular.}
As in the proof of Proposition~\ref{p:main2}, we can w.l.o.g. suppose that $\mathcal S_\kappa = (F_\alpha,f_{\beta\alpha}:F_\alpha\to F_\beta\mid \alpha<\beta<\kappa)$ and, using Construction~\ref{concoker}, we deduce that $\Coker(f_{\beta\alpha})$ is almost $(\Cogen_*(C),\lambda)$-projective for each regular $\lambda>\aleph_0$ and ordinals $\alpha<\beta<\kappa$. Hence $\mathfrak S_C(\Coker(f_{\beta\alpha}))$ holds by the inductive hypothesis, and subsequently $\mathfrak S_C(M)$ holds as well.

\smallskip

$(3)\Longrightarrow (1)$. Put $J = PE(C)$. Then $\Cogen_*(C) \subseteq \Cogen_*(J)$, and $\mathfrak S_J(M)$ holds by Remark. By induction on $\pres(M)$, we prove that $\Ext _R^1(M,\Cogen_*(C)) = 0$ whenever $\mathfrak S_J(M)$ and there is a directed system $\mathcal T$ witnessing that $M$ is almost $(J,\aleph_1)$-projective. If $M$ is countably presented, we can use \cite[Proposition~4.3]{S}.

Otherwise, \chge{since $M$ is $J$-stationary by Remark, we can use Lemma~\ref{l:Cext} to} obtain a system $\mathcal S$, consisting of $J$-stationary modules, such that $\mathcal S$ witnesses the $(J,\aleph_1)$-projectivity of $M$ (if $\mathcal S\not\subseteq {}^\perp J$, just \chge{intersect} $\mathcal S$ and $\mathcal T$ using Construction~\ref{conmerge}). By \cite[Proposition 4.3]{S}, $\mathcal S$ witnesses also that $M$ is almost $(\Cogen_*(J),\aleph_1)$-projective. Furthermore, using the pure-injectivity of~$J$ \chge{(so that ${^\perp J}$ is closed under direct limits by \cite[Theorem 6.19]{GT}), it follows from the remark after Lemma~\ref{l:Cext}} that the system $\mathcal S^\lambda$ from Definition~\ref{d:inducesys} witnesses the $(J,\lambda)$-projectivity of~$M$ for each regular $\lambda>\aleph_0$.

If $\pres(M)$ is singular, we can w.l.o.g. assume that $\mathcal S^\lambda$ consists of modules satisfying $\mathfrak S_C$ for all \chge{successor cardinals} $\aleph_0<\lambda<\pres(M)$; just use Construction~\ref{conmerge} and $\mathfrak S_C(M)$.
\chge{Since each $N\in\mathcal S^\lambda$ also satisfies $\mathfrak S_J$ and, thanks to Observation~\ref{o:observ}, is almost $(J,\aleph_1)$-projective, we have $\Ext^1_R(N,\Cogen_*(C))=0$ by the inductive hypothesis and, consequently, $\mathcal S^\lambda$ witnesses almost $(\Cogen_*(C),\lambda)$-projectivity of $M$.} Whence we can use Lemma~\ref{l:singmod} to conclude that $\Ext _R^1(M,\Cogen_*(C)) = 0$.

Now let $\theta = \pres(M)$ be regular uncountable. Take $\mathcal M$ from the definition of $\mathfrak S_C(M)$ and \chge{intersect} it with $\mathcal S^\theta$ using Construction~\ref{conmerge} (we can assume that $M_0 = 0$ is in the resulting system). Let $\alpha<\beta<\theta$ be arbitrary. Then $\Coker(f_{\beta\alpha})\in {}^\perp J$, since $f_{\beta\alpha}$ is $J$-injective and $M_\beta\in {}^\perp J$. We apply Construction~\ref{concoker} with $\lambda = \aleph_1$ to build, from $\mathcal S$, an $\aleph_1$-continuous directed system $\mathcal T^\prime$ consisting of countably presented modules from ${}^\perp J$ such that $\varinjlim\mathcal T^\prime = \Coker(f_{\beta\alpha})$. \chge{Since this shows that $\Coker(f_{\beta\alpha})$ is almost $(J,\aleph_1)$-projective, and also $\mathfrak S_J(f_{\beta\alpha})$ by the choice of $\mathcal M$, we may use the inductive hypothesis to obtain} that $\Coker(f_{\beta\alpha})\in {}^\perp\Cogen_*(C)$. In particular $M_\beta\in {}^\perp\Cogen_*(C)$ (which is the case $\alpha = 0$). Lemma~\ref{l:cogen} gives that $f_{\beta\alpha}$ is $\Cogen_*(C)$-injective. Since $\alpha,\beta$ were chosen arbitrarily, we conclude that $\Ext _R^1(M,\Cogen_*(C)) = 0$ using Lemma~\ref{l:eklof}.

\chge{To summarize, starting with an almost $(J,\aleph_1)$-projective module $M$ satisfying (3), we have shown that $M\in{^\perp\B}$ and is almost $(\B,\aleph_1)$-projective for $\mathcal B = \Cogen_*(C)$.
Then (1) follows by Lemma~\ref{l:filter-closed}}
\end{proof}

\begin{cor} Let $R$ be a ring, $C$ be an $R$-module and $\kappa = |R|+\aleph_0$. Then $C$ is $\Sigma$-cotorsion \iff $C^{(\kappa)}$ is cotorsion, \iff $\mathfrak S_C(M)$ holds for each ($\kappa$-presented) flat $R$-module $M$.
\end{cor}

\begin{proof} Apply Proposition~\ref{p:stationary} for each $\kappa$-presented flat module $M$, and use the fact that, under our assumptions, each flat module has a filtration with consecutive factors flat and $\kappa$-presented. Notice also that each flat module is almost $(PE(C),\aleph_1)$-projective since it is the direct limit of an $\aleph_1$-continuous directed system consisting of flat modules.
\end{proof}

\begin{exm} The implication $(3)\Longrightarrow (2)$ in Proposition~\ref{p:stationary} does not hold without the additional assumption. To see this, let $R$ be an $\aleph_0$-noetherian ring and $C$ be a $\Sigma$-pure injective $R$-module which is not injective (e.g. $R = \Z$ and $C$ is nonzero finite). By Baer injectivity test, there exists a (cyclic) countably presented module~$M$ such that $\Ext_R^1(M,C)\neq 0$. However, $\mathfrak S_C(M)$ follows from the $\Sigma$-pure injectivity of $C$, cf.\,\cite[Lemma~5.1]{S}. Thus $(3)$ holds and $(2)$ does not hold true.
\end{exm}

\begin{exm} The condition $(2)$ in Proposition~\ref{p:stationary} does not imply $M\in {}^\perp\bar C$. Indeed, let $R$ be a ring which is not right coherent and $C$ an injective cogenerator of $\ModR$. Then $\bar C$ contains a module $D$ which is not absolutely pure. Let $M$ be a finitely presented module such that $\Ext_R^1(M,D)\neq 0$, whence $M\not\in {}^\perp\bar C$. However, we see that the condition $(2)$ from Proposition~\ref{p:stationary} holds true since $C^{(\kappa)}$ is absolutely pure for any $\kappa$ (and so $\Ext_R^1(M,C^{(\kappa)}) = 0$).
\end{exm}

\section{Projectively coresolved Gorenstein flat modules}
\label{sec:Gor}

\chge{In the next two sections we study implications of our set-theoretical tools to Gorenstein homological algebra. The highlight of this section is the fact that, for any ring, Gorenstein flat covers always exist. We also construct new model structures on the category of modules over any ring which refines the Gorenstein AC-projective model structure from~\cite{BGH} (in that the model structure from \cite{BGH} is a Bousfield localization of ours).
	
\smallskip}

By a \emph{projectively coresolved Gorenstein flat module}, or a \emph{$PGF$-module} for short, we mean a syzygy module in an acyclic complex

\begin{equation}\label{eq:pgf}
\dotsb \to P^{-1}\to P^0\to P^{1}\to P^{2}\to\dotsb
\end{equation}
consisting of projective modules which remains exact after tensoring by arbitrary injective left $R$-module. 
We denote the class of all such modules by $\PGF$. 
For a comparison, recall that a module is \emph{Gorenstein projective} if it is a syzygy in an acyclic complex \eqref{eq:pgf} consisting of projective modules which remains exact after applying the functor $\Hom_R(-,P)$ for arbitrary projective module $P$.

Finally, a module is \emph{Gorenstein flat} if it is a syzygy in an acyclic complex \eqref{eq:pgf} consisting of flat modules which remains exact after tensoring by arbitrary injective left $R$-module. We denote the classes of Gorenstein projective and Gorenstein flat modules by $\mathcal{GP}$, $\mathcal{GF}$ respectively.

\smallskip

We are going to apply results from the two previous sections to prove a general statement which yields $\PGF\subseteq {}^\perp \overline{R_R}$. Consequently, each module from $\PGF$ is Gorenstein projective. On the other hand, each Gorenstein AC-projective module in the sense of Bravo, Gillespie and Hovey (cf. \cite{BGH}) is a PGF-module. Moreover, if $R$ is left coherent, then $\PGF$ is precisely the class of all Gorenstein AC-projective modules since $\overline{R_R}$ coincides with the class of all level (or equivalently flat) right $R$-modules in this case.


\smallskip

The next lemma is folklore.

\begin{lem}\label{l:red} Let $P^\bullet$ be a complex. If we sum up all its shifts, we get a $1$-periodic complex which induces an exact sequence $\chge{0 \longrightarrow H \longrightarrow} M \longrightarrow \bigoplus _{n\in \Z}P^n \longrightarrow M \longrightarrow 0$ \st \chge{$H\cong \bigoplus_{n\in\Z} H^n(P^\bullet)$ and} $M\cong \bigoplus_{n\in\Z} P^n/B^n(P^\bullet)$ where $B^n(P^\bullet)$ denotes the $n$th coboundary module of $P^\bullet$ \chge{ and $H^n(P^\bullet)$ the $n$th cohomology module}.

In particular, if $K\in\PGF$, then $K$ is a direct summand in a module $M$ such that $M\cong P/M$ with $P$ projective and $\Tor_1^R(M,I) = 0$ for each injective left $R$-module~$I$.
\end{lem}

We believe that the following statement could be of independent interest. Combined with the lemma above, it generalizes \cite[Theorem A.6]{BGH}.

\begin{prop} \label{p:key} Let $I$ be a left $R$-module and $$M\overset{f}{\longrightarrow} P \overset{}{\longrightarrow} M \longrightarrow 0$$ an exact sequence of modules with $P\in {}^\perp\overline{I^c}$. Assume that $f\otimes_R I^\theta$ is injective for all cardinals $\theta$. Then $f$ is $\overline{I^c}$-injective; in particular, $M\in {}^\perp\overline{I^c}$.
\end{prop}

\begin{proof} Put $C = I^{c}$. By \cite[Lemma 4.2]{S}, $P$ is (strict) $C$-stationary, equivalently $I$-Mittag-Leffler (cf. \cite[Corollary 2.6]{H}). By our assumption on $f\otimes_R I^\theta$, $M$ is $I$-Mittag-Leffler \chge{and $C$-stationary} as well.

Let $\kappa$ denote the least infinite cardinal such that $P$ is $\kappa$-presented. Notice that $M$ is $\kappa$-presented as well. We are going to prove by induction on $\kappa$ that $M\in {}^\perp\Cogen_*(C)$. For $\kappa = \aleph_0$, it holds: $f$ is $C$-injective (by our assumption for $\theta = 1$ and the well-known relations between $\otimes$ and $\Hom$) which yields $\Ext_R^1(M,C) = 0$ since $\Ext_R^1(P,C) = 0$, and it remains to use \cite[Proposition 4.3]{S}.

Now assume that $\kappa$ is regular uncountable. Using Lemma~\ref{l:Cext} for $\lambda = \kappa$, we obtain a $\kappa$-continuous directed system $\mathcal L = \mathcal L_\kappa$ consisting of $C$-stationary $<\kappa$-presented modules such that $\varinjlim\mathcal L = M$. Passing to a cofinal subsystem, we can w.l.o.g. assume that $\mathcal L$ is well-ordered by $\kappa$. We have $\mathcal L = (L_\alpha,f_{\beta\alpha}\mid \alpha<\beta<\kappa)$. For each $\alpha<\kappa$, let $f_\alpha\colon L_\alpha\to M$ be the canonical colimit map.

By Theorem~\ref{t:main1}, we obtain a system $\mathcal S=\mathcal S_\kappa$ witnessing that $P$ is $(\bar C,\kappa)$-projective; as before, we can assume that it is well-ordered, so $\mathcal S = (S_\alpha,g_{\beta\alpha}\mid \alpha<\beta<\kappa)$. Again, let $g_\alpha\colon S_\alpha\to \chge{P}$ denote the canonical colimit map for each $\alpha<\kappa$. Possibly dropping some indices, Construction~\ref{concoker} provides us with a $\kappa$-continuous directed system $\mathcal U = (u_\alpha\colon L_\alpha \to S_\alpha, (f_{\beta\alpha},g_{\beta\alpha})\mid \alpha<\beta<\kappa)$ of morphisms with $\varinjlim\mathcal U = f$. 
We also get the well-ordered directed system $\mathcal K = (\Coker(u_\alpha), h_{\beta\alpha}\mid \alpha<\beta<\kappa)$ with canonically defined morphisms. It follows that $M =\varinjlim \mathcal K$.

\chge{If we apply Construction~\ref{conmerge} to $\mathcal L$ and $\mathcal K$, we can assume, by possibly passing to a closed and unbounded subset of $\kappa$,} that $L_\alpha = \Coker(u_\alpha)$ and that the canonical colimit map $\Coker(u_\alpha)\to M$ equals $f_\alpha$ for each $\alpha<\kappa$, as well as \chge{that} $h_{\beta\alpha} = f_{\beta\alpha}$ for all $\alpha<\beta<\kappa$.

Let $\theta$ be a cardinal and put $D= (I^\theta)^c$. We claim that $u_\alpha\otimes_R I^\theta$ is injective for every $\alpha<\kappa$, or equivalently: $u_\alpha$ is $D$-injective. To see this, notice that $f$ is $D$-injective by the hypothesis and $f_\alpha$ is $D$-injective by the statement of Lemma~\ref{l:Cext} for each $\alpha<\kappa$ (note that $D\in\bar C$). The claim follows from the identity $ff_\alpha = g_\alpha u_\alpha$.

Since $S_\alpha\in {}^\perp\bar C$, we may use the inductive hypothesis to deduce that $L_\alpha\in {}^\perp\Cogen_*(C)$ for each $\alpha<\kappa$. Considering the following diagram with exact rows and columns 

$$\begin{CD}
 0  @. 0  @.   0 \\
  @AAA				@AAA		@AAA \\
	\Coker (f_{\alpha+1,\alpha}) 	@>{\bar u}>>		\Coker (g_{\alpha+1,\alpha}) @>>> 	\Coker (f_{\alpha+1,\alpha})		@>>> 0	\\
	@AAA				@AAA	@AAA		\\
	L_{\alpha+1} @>{u_{\alpha+1}}>>	S_{\alpha+1} @>>>	 	L_{\alpha+1}	@>>>	0	\\
	@A{f_{\alpha+1,\alpha}}AA				@A{g_{\alpha+1,\alpha}}AA	@A{f_{\alpha+1,\alpha}}AA		\\
	L_\alpha 	@>{u_\alpha}>>	S_\alpha	@>>> 	L_\alpha	@>>>	0,
\end{CD}$$

\medskip

\noindent we see that $\Coker (g_{\alpha+1,\alpha})\in {}^\perp\bar C$ holds, by the property of $\mathcal S$, and that $\bar u\otimes_R I^\theta$ is injective for any $\theta$: indeed, just apply the functor $-\otimes_R I^\theta$ on the diagram and use the $3\times 3$ lemma. Once more using the inductive hypothesis, we infer that $\Coker(f_{\alpha+1,\alpha})\in {}^\perp\Cogen_*(C)$ for each $\alpha<\kappa$. By Lemma~\ref{l:cogen}, $f_{\alpha+1,\alpha}$ is $\Cogen_*(C)$-injective, and Lemma~\ref{l:eklof} yields $M\in {}^\perp\Cogen_*(C)$.

\medskip

Now let $\kappa$ be singular. For each $\lambda<\kappa$ regular uncountable, there are systems $\mathcal L_\lambda, \mathcal S_\lambda$ such that $\mathcal S_\lambda$ witnesses that $P$ is $(\bar C,\lambda)$-projective and $\mathcal L_\lambda$ is provided by Lemma~\ref{l:Cext}. By a similar argument as in the regular step, combining Construction~\ref{concoker} and \ref{conmerge}, we can obtain a $\lambda$-continuous directed system $\mathcal U_\lambda = (u_j\colon L_j\to S_j \mid j\in J)$ of morphisms with $\varinjlim\mathcal U = f$, $L_j\in\mathcal L_\lambda$, $S_j\in\mathcal S_\lambda$, and such that the induced directed system $(\Coker(u_j) \mid j\in J)$ of modules is identical with $\mathcal L_\lambda$.

As in the regular step, we observe that $u_j\otimes_R I^\theta$ is injective for each $\theta$ and $j\in J$. By the inductive hypothesis, we get that $L_j\in {}^\perp\Cogen_*(C)$ for all $j\in J$. Using Lemma~\ref{l:singmod} with $\mathcal D = \Cogen_*(C)$ and $\nu = \aleph_0$, we obtain $M\in {}^\perp\Cogen_*(C)$. \chge{This finishes the induction.}
\smallskip

Finally, we show that $f$ is $\bar C$-injective. Note that $\bar C$ is just the closure of $\Cogen_*(C)$ under pure-epimorphic images. Let $h\colon D\to E$ be a pure-epimorphism with $D\in\Cogen_*(C)$. We have to check that every morphism $m\colon M\to E$ factorizes through~$f$. However, $\Ker(h)\in\Cogen_*(C)$ yields the existence of $n\colon M\to D$ such that $hn = m$. Since \chge{$M\in {}^\perp\Cogen_*(C)$ and $f$ is $C$-injective}, $f$ is $\Cogen_*(C)$-injective by Lemma~\ref{l:cogen}. We can thus factorize the morphism $n$ through $f$. The composition of the resulting map $r\colon P\to D$ with $h$ is the desired factorization.
\end{proof}

\begin{exm} Let \chge{$0\longrightarrow M\overset{f}\longrightarrow P\longrightarrow M\longrightarrow 0$ be a short exact sequence}, $R$ right coherent and $I = {}_RR$. The assumptions of Proposition~\ref{p:key} hold whenever $P$ is an fp-projective module, i.e.\ $P\in {}^\perp\overline{I^c}$\chge{, where $\overline{I^c}$ is none other than the class of absolutely pure (i.e.\ fp-injective) modules}. As a result, we get that $M$ is fp-projective as well.

Using the reduction from Lemma~\ref{l:red}, we obtain a generalization of \cite[Theorem~3.6]{G1}: over a right coherent ring, every syzygy module in an acyclic complex of fp-projective modules is fp-projective. By the dual reasoning to the one in the proof of \cite[Theorem 4.3]{BIE}, we get the equality dw(fpProj) $=$ dg(fpProj) over a right coherent ring.
\end{exm}

We are now ready to prove that

\begin{thm} \label{t:pgf} Every projectively coresolved Gorenstein flat module belongs to the~class ${}^\perp\overline{R_R}$, in particular it is Gorenstein projective.
\end{thm}

\begin{proof} Let $K\in\PGF$ be arbitrary. By Lemma~\ref{l:red}, $K$ is a direct summand of a~module $M$ \st $M\cong P/M$ with $P$ projective and $\Tor_1^R(M,I)= 0$ for any injective left $R$-module $I$. We use Proposition~\ref{p:key} for $I= (R_R)^c$ and $f:M\hookrightarrow P$ the inclusion.
\end{proof}

Recall that, over any ring $R$, we have $\mathcal{FL}\subseteq\overline{R_R}$. We sum up what we achieved in the corollary below.\footnote{One can alternatively use \cite[Theorem A.6]{BGH} to prove Corollary~\ref{c:pgf}.} Note that $\overline{(R_R)^c}$ is the class of all absolutely pure left $R$-modules \iff $R$ is left coherent, \iff $\overline{R_R}$ coincides with the~class of all flat modules.
\begin{cor} \label{c:pgf} Let $K$ be a right $R$-module. Then the following conditions are equivalent:
\begin{enumerate}
  \item $K$ is projectively coresolved Gorenstein flat;
	\item $K$ is a syzygy in a long exact sequence \eqref{eq:pgf} of projective modules which stays exact after applying the functor $\Hom_R(-,F)$ for any $F\in\ModR$ from the definable closure of $\{R_R\}$;
	\item $K$ is a syzygy in a long exact sequence \eqref{eq:pgf} of projective modules which stays exact after applying the functor $-\otimes_R I$ for any $I\in\RMod$ from the definable closure of $\{(R_R)^c\}$.
\end{enumerate}
\end{cor}


We can prove more than just $\PGF\subseteq {}^\perp\overline{R_R}$. We will see \chge{in Theorem~\ref{t:Kaplansky} that $\PGF$ forms the left-hand class of a hereditary cotorsion pair which is generated by a set, hence complete.}

To do so, we will need to construct a filtration of an exact sequence of the form $M\overset{f}\longrightarrow P\overset{g}\longrightarrow M\longrightarrow 0$ by exact sequences of the same form. Ignoring the exactness for the moment, we will view the sequences as representations of the quiver
\[
Q\colon\quad
\xymatrix{
p \ar@/^/[r]^-g & m \ar@/^/[l]^-f
}
\]
in the category $\ModR$ which are bound by the relation $gf=0$. The category of such representations is equivalent to the category of left modules over the path algebra $R_1 = R^\op Q/(gf)$ (see~\cite[Proposition III.1.7]{ARS}, whose proof is valid also in our situation).

\begin{rem}
The ring $R_1$ is isomorphic to the subring of the full matrix ring $M_3(R^\op)$ formed by matrices of the form
\[
\begin{pmatrix}
a & b & c \\
0 & a & 0 \\
0 & d & e \\
\end{pmatrix}.
\]
This isomorphism can be obtained from the left action $R_1 \longrightarrow \End_R(M)$ on the faithful left $R_1$-module corresponding to the representation
\[
\xymatrix{
R\oplus R \ar@/^1em/[r]^-{\left(\begin{smallmatrix}0&1\end{smallmatrix}\right)} & R. \ar@/^1em/[l]^-{\left(\begin{smallmatrix}1\\0\end{smallmatrix}\right)}
}
\]
\end{rem}

Let us denote by $\RE$ the full subcategory of $\RoneMod$ formed by all modules corresponding to the representations 
\begin{equation} \label{eq:rep}
\xymatrix@1{P \ar@/^/[r]^-g & M \ar@/^/[l]^-f}
\end{equation}
with $M\overset{f}\longrightarrow P\overset{g}\longrightarrow M\longrightarrow 0$ right exact.

\begin{lem} \label{lem:RE}
If $M,P$ are $\kappa$-presented in $\ModR$ and $U\in\RoneMod$ corresponds to a representation of the form~\eqref{eq:rep}, then $U$ is $\kappa$-presented in $\RoneMod$.
The class $\RE$ is closed in $\RoneMod$ under cokernels and extensions.
\end{lem}

\begin{proof}
The first assertion follows from the fact that $R_1$ is free of finite rank as a~right $R$-module.
The closure under cokernels comes from the fact that the cokernel of a~map between right exact sequences is right exact. For the closure under extensions, we consider the elements of $\RoneMod$ as cochain complexes of right $R$-modules concentrated in degrees $-1, 0, 1$ and use the long exact sequence of cohomologies.
\end{proof}

Now, given an infinite cardinal $\nu$, we say that a ring $R$ is \emph{right $\nu$-coherent} if each $\nu$-generated right ideal of $R$ is $\nu$-presented. If every right ideal of $R$ is $\nu$-generated, we say that $R$ is \emph{right $\nu$-noetherian}.

\begin{prop} \label{p:filtr} Let $\nu$ be an infinite cardinal, $W$ an injective cogenerator in $\ModR$ and $I$ a left $R$-module. Let moreover $\mathcal D$ be a definable subclass of $\ModR$ \st $I^c,W\in\mathcal D$. Assume that $\mathcal E:M\overset{f}\rightarrow P\rightarrow M\rightarrow 0$ is an exact sequence with $P\in {}^\perp\mathcal D$ and $f\otimes_R I^\theta$ injective for any cardinal $\theta$. Let moreover
\begin{enumerate}
	\item $R$ be a right $\nu$-noetherian ring, or
	\item $R$ be a right $\nu$-coherent ring, $f$ be a monomorphism and $P$ be projective, or
	\item $f\otimes_R R^\theta$ be injective for all cardinals $\theta$ $($in this case, set $\nu = \aleph_0)$.
\end{enumerate}
Then there exists a filtration $\mathfrak F = (\mathcal E_\alpha\colon M_\alpha\overset{u_\alpha}\rightarrow P_\alpha\rightarrow M_\alpha\rightarrow 0\mid \alpha\leq\sigma)$ of $\mathcal E$ where $P_{\alpha+1}/P_\alpha\in {}^\perp\mathcal D$ and $M_{\alpha+1}/M_\alpha$ is $\nu$-presented with $M_{\alpha+1}/M_\alpha \in {}^\perp\overline{I^c}$ for each $\alpha<\sigma$.
\end{prop}

\begin{proof} \chge{Let $\kappa$ be the least infinite cardinal such that $P$ (and hence also $M$) is $\kappa$-presented. For the sake of nontriviality, let us assume that $\kappa>\nu$.} We follow the lines of the proof of Proposition~\ref{p:key} proving by induction \chge{on $\kappa$} the existence of the filtration $\mathfrak F$. 
The only difference is the choice of the systems $\mathcal L_\lambda$, $\mathcal S_\lambda$:

Using Theorem~\ref{t:main1}, we let $\mathcal S_\lambda$ be a system witnessing $(\mathcal D,\lambda)$-projectivity of~$P$. Since $W\in\mathcal D$, we can w.l.o.g.\ assume that $\mathcal S_\lambda$ consists of submodules of $P$ and inclusions.

It is easy to observe that, if $(1)$ or $(2)$ holds, every $\eta$-generated submodule of $M$ is $\eta$-presented whenever $\eta \geq \nu$. Put $C = I^c$ in this case, and $C = I^c\oplus W$ if $(3)$ holds true. By Proposition~\ref{p:key}, we know that $M\in {}^\perp\bar C$, thus we can consider a~system $\mathcal L_\lambda$ witnessing that $M$ is $(\bar C,\lambda)$-projective (using Theorem~\ref{t:main1} again). For $\lambda>\nu$, we can w.l.o.g. assume that the elements of $\mathcal L_\lambda$ are, in fact, $<\lambda$-presented submodules of $M$.
\chge{Indeed, in cases (1) and (2) we can intersect $\mathcal L_\lambda$ with the system of all $<\lambda$ generated submodules of $M$ using Construction~\ref{conmerge}, while in case (3) we use that $W\in \bar C$.}

Furthermore, we can, as in the proof of Proposition~\ref{p:key}, without loss of generality express the morphism $f\colon M \to P$ as the direct limit 
a $\lambda$-continuous directed system $\mathcal U_\lambda = (u_j\colon L_j\to S_j \mid j\in J)$ of morphisms with $\varinjlim\mathcal U_\lambda = f$, $L_j\in\mathcal L_\lambda$, $S_j\in\mathcal S_\lambda$, and such that the induced directed system $(\Coker(u_j) \mid j\in J)$ of modules is identical with $\mathcal L_\lambda$. If $U\in\RoneMod$ corresponds to the representation as in~\eqref{eq:rep}, we have in particular expressed $U$ as a $\lambda$-continuous direct limit of left $R_1$-submodules, which are all contained in $\RE$.

If $\kappa=\lambda$ is regular, we obtain a filtration with $<\kappa$-presented consecutive factors, which we refine by \chge{applying the inductive hypothesis to these factors, to get} the desired $\mathfrak F$. Note that the consecutive factors in $\mathfrak F$ are also contained in $\RE$ by Lemma~\ref{lem:RE}.

Finally, if $\kappa$ is singular, we apply the classical singular compactness theorem to $U\in\RoneMod$, cf.\ \cite[Theorem~7.29]{GT}, to get $\mathfrak F$.
\end{proof}

Note that, over a right coherent ring $R$, the condition $(3)$ in the statement of Proposition~\ref{p:filtr} is equivalent to the injectivity of the map $f$.

\chge{Now we discuss some basic closure properties of the classes of ordinary and of projectively coresolved Gorenstein flat modules.}

\begin{lem} \label{l:PGF} Let $B\in\PGF$ and $C\in\mathcal{GF}$. Then $A\in\mathcal{GF}$ provided that it fits into one of the following two short exact sequences:
\begin{enumerate}
	\item $0\to C \to A\to B\to 0,$
	\item $0\to A \to B\to C\to 0.$
\end{enumerate}
Moreover, if $C\in\PGF$, then $A$ is a PGF-module as well. Finally, $\PGF$ is closed under transfinite extensions.
\end{lem}

\begin{proof} The proof mimics the one of \cite[Theorem 2.5]{Ho}. We give just a brief sketch.
Let us denote by $i$ the monomorphism $C\to A$ and by $f$ the epimorphism $A\to B$ in the short exact sequence from $(1)$. There are short exact sequences $0\to B \overset{h}{\to} P\to B^\prime\to 0$ and $0\to C\overset{g}{\to} F\to C^\prime\to 0$ with $P$ projective, $F$ flat, $B^\prime\in\PGF$ and $C^\prime$ Gorenstein flat. We define a morphism $m_1:A\to F$ as a~factorization of $g$ through $i$; this is possible, since $\FL\subseteq\PGF^\perp$ by Theorem~\ref{t:pgf}. Further, we let $m_2 = hf$ and let $m:A\to F\oplus P$ be determined by $m_1$ and $m_2$.

Then $m$ is a monomorphism, $\Coker(m)$ is an extension of $C^\prime$ by $B^\prime$, and we can repeat the process to obtain an acyclic complex of flat modules where all syzygies, including $A$, belong to $\Ker\Tor^R_n(-,I)$ for all injective left $R$-modules $I$ and $n>0$ (see also \cite[Lemma 2.4(2)]{B}).

To prove the alternative $(2)$, we take the pushout of $h$ and the epimorphism $B\to C$ from $(2)$. In the resulting short exact sequence $0\to A \to P \to H\to 0$, we see, by $(1)$, that $H\in\mathcal{GF}$ since it is an extension of $C$ by $B^\prime$. By \cite[Lemma 2.4(3)]{B}, it follows that $A$ is Gorenstein flat.

The proof of the moreover clause is analogous. In this case, we have $F$ projective and $C^\prime\in\PGF$.

Finally, the construction used in the proof of the alternative $(1)$ for $C\in\PGF$ can be iterated to show that a transfinite extension of PGF-modules is again a~syzygy in an acyclic complex of projective modules. The rest follows from Eklof lemma.
\end{proof}

Recall that a class $\mathcal W\subseteq\ModR$ is called \emph{thick} provided that it is closed under direct summands, extensions, and taking kernels of epimorphisms and cokernels of monomorphisms.

\begin{thm} \label{t:Kaplansky} $\mathfrak{PGF}=(\PGF,\PGF^\perp)$ is a complete \chge{hereditary} cotorsion pair with $\PGF^\perp$ thick. If $R$ is right $\aleph_0$-coherent, then $\mathfrak{PGF}$ is of countable type and each module from $\PGF$ is filtered by countably presented PGF-modules.

Moreover, $R$ is right perfect \iff $\PGF$ coincides with the class $\mathcal{GF}$ of all Gorenstein flat modules.
\end{thm}

\begin{proof} Let $(\mathcal A,\mathcal B)$ be the cotorsion pair generated by a representative set \chge{$\A_0$} of $\nu$-presented modules from $\PGF$; here $\nu$ is as in Proposition~\ref{p:filtr}(2). Then $\A$ consists precisely of direct summands of $\A_0$-filtered modules by \cite[Corollary 6.14]{GT}. Using this, Proposition~\ref{p:filtr}, the Eklof Lemma and Lemma~\ref{l:red}, we see that $\PGF\subseteq\mathcal A$. Since $R_R\in\PGF$ and $\PGF$ is closed under taking transfinite extensions by Lemma~\ref{l:PGF} and direct summands by \cite[Proposition 1.4]{Ho}, we get $\mathcal A = \PGF$. The completeness of $\mathfrak{PGF}$ follows by \cite[Corollary 6.14 and Theorem 6.11]{GT}. The thickness of $\PGF^\perp$ stems from the definition of a PGF-module. \chge{The statement about right $\aleph_0$-coherent rings follows from \cite[Theorem 7.13]{GT}.}

The only-if part of the moreover clause is trivial. The other implication follows from the fact that $\PGF \cap \PGF^\perp$ is precisely the class $\mathcal P_0$ of all projective modules which stems from the thickness of $\PGF^\perp\supseteq \mathcal P_0$. Note that flat modules belong to $\PGF^\perp$ by Theorem~\ref{t:pgf}.
\end{proof}

\begin{exm} Let $R$ be an Artin algebra. Then the cotorsion pair $\mathfrak{PGF}$ is of countable type. Moreover, it is of finite type (i.e. generated by finite dimensional PGF-modules) if and only if $R$ is virtually Gorenstein, cf.\,\cite[Theorem 5]{BK}. Of course, in this case $\PGF = \mathcal{GF} = \mathcal{GP}$.
\end{exm}

\begin{rem}
Let $R$ be a right $\aleph_0$-coherent ring. It is well known (cf. \cite[Lemma~3.4]{WL}) that each countably presented Gorenstein projective $R$-module is (projectively coresolved) Gorenstein flat. We have thus identified the class $\PGF$ as the subclass of Gorenstein projective modules filtered by countably presented Gorenstein projective modules.

The question whether Gorenstein projective modules are necessarily (projectively coresolved) Gorenstein flat remains open. However, we have manifested that, to answer this question in positive over $R$ right $\aleph_0$-coherent, it is necessary to show that each Gorenstein projective module is filtered by countably presented Gorenstein projective modules.

\smallskip

Our results suggest that the notion of a projectively coresolved Gorenstein flat module could serve as an alternative definition of a Gorenstein projective module over any ring. In fact, Theorem~\ref{t:Kaplansky} can be viewed as a Gorenstein analogue of the Kaplansky theorem on the decomposition of projective modules.

Finally, if it happens that $\mathcal{GP}\subseteq\mathcal{GF}$, then necessarily $\mathcal{GP} = \PGF$. 
However, it is not clear whether $\mathcal{GF}\cap \mathcal{GP} \subseteq \PGF$ since we do not know if $\mathcal{FL}\cap\mathcal{GP} = \mathcal P_0$ holds true, cf. \cite[Question 2.8]{BIE}.
\end{rem}

The rest of this section is devoted to clarifying the relation between $PGF$-modules and Gorenstein flat modules. We obtain a description of the class $\mathcal{GF}$ which implies that it is always closed under extensions, regardless of the ring $R$ \chge{(hence \emph{every} ring is GF-closed in the terminology of \cite{B})}.

\begin{thm} \label{t:GF} Let $M$ be a module. Then the following conditions are equivalent:
\begin{enumerate}
	\item $M$ is Gorenstein flat.
	\item There is a short exact sequence $$0\to K \to L\to M\to 0$$ with $K\in\FL$ and $L\in\PGF$ which remains exact after applying the functor $\Hom_R(-,C)$ for any (flat) cotorsion module $C$.
  \item $\Ext_R^1(M,C) = 0$ for all cotorsion modules $C\in\PGF^\perp$.
	\item There is a short exact sequence $$0\to M \to F\to N\to 0$$ with $F$ flat and $N\in\PGF$.
\end{enumerate}
In particular, we get $\mathcal{GF}\cap\PGF^\perp = \FL$.
\end{thm}

\begin{proof} $(1)\Longrightarrow (2)$. The module $M$ is a syzygy 
in an exact complex $F^\bullet$ consisting of flat modules which remains exact after applying the functor $-\otimes_R I$ for any $I\in\RMod$ injective. Consider the complete cotorsion pair $(\Ch{\mathcal P_0},\Ch{\mathcal P_0}^\perp)$ in $\ChR$, where $\Ch{\mathcal P_0}$ denotes the class of all complexes of projective modules, cf.\;\cite[Theorem 4.5]{BEIJR}. Let $0\to G^\bullet \overset{h^\bullet}{\to} P^\bullet \to F^\bullet \to 0$ be a special $\Ch{\mathcal P_0}$-precover of~$F^\bullet$\chge{, i.e.\ $P^\bullet\in\Ch{\mathcal P_0}$ and $G^\bullet\in\Ch{\mathcal P_0}^\perp$. It follows that the} complex $G^\bullet$ is exact, whence $P^\bullet$ is exact as well. For each $i\in\Z$, we obtain the following commutative diagram of modules with exact rows and columns

$$\begin{CD}
 @.  0  @. 0  @.   0 \\
  @.  @AAA				@AAA		@AAA \\
	0 @>>> M^i 	@>>>	F^i	@>>> 	M^{i+1}	@>>>	0\\
	@.  @AAA				@AAA	@AAA		\\
	0 @>>> L^i 	@>>>	P^i	@>>> 	L^{i+1}	@>>>	0\\
	@.  @A{f^i}AA				@A{h^i}AA	@AAA		\\
	0 @>>> K^i 	@>{g^i}>>	G^i	@>>> 	K^{i+1}	@>>>	0 \\
  @.  @AAA				@AAA		@AAA \\
 @.  0  @. 0  @.   0
\end{CD}$$
where $h^i$ is pure since $F^i$ is flat. It follows that $G^\bullet$ consists of flat modules and as such it is a direct limit of complexes from $\Ch{\mathcal P_0}$\chge{; see~\cite[Lemma 8.4]{Nee}}. On the other hand $G^\bullet\in \Ch{\mathcal P_0}^\perp$, and so each morphism from a finitely presented complex to $G^\bullet$ is null-homotopic. In particular, $G^\bullet$ is pure exact and $K^i$ is flat \chge{(see also \cite[Theorem 8.6]{Nee})}. As $F^\bullet$ and $G^\bullet$ remain exact after applying $-\otimes_R I$ for an injective $I\in\RMod$, the same holds for $P^\bullet$, which yields $L^i\in\PGF$. Finally, $f^i$ is $C$-injective for any $C$ cotorsion since $g^i$ and $h^i$ are (we use that $K^{i+1}$ and $F^i$ are flat).

$(2)\Longrightarrow (3)$. Let $0\to K\overset{f}{\to} L\to M\to 0$ be a short exact sequence with $K$ flat and $L\in\PGF$ \st $f$ is $\EC\cap\FL$-injective. The result will follow, once we show that any morphism $g\colon K \to C$ with $C$ cotorsion (from $\PGF^\perp$) factorizes through $f$. Let us form a~flat cover $\pi\colon F \to C$. Then there exists a map $h\colon K \to F$ \st $\pi h = g$ since $K$ is flat and $\Ker(\pi)$ cotorsion. The flat module $F$ is also cotorsion, whence we can factorize $h$ through~$f$. The composition of the resulting map with $\pi$ is the desired factorization of $g$.

$(3)\Longrightarrow (4)$. First, we show that $M$ is a pure-epimorphic image of a PGF-module. We start with a special $\PGF$-precover $\pi\colon P \to M$ of the module $M$. Put $K = \Ker(\pi)$ and let us consider the following pushout diagram 
$$\begin{CD}
 @. 0  @. 0  @.   \\
  @. @AAA				@AAA		@. \\
	@. G 	@=		G @.	@.	\\
	@. @AAA				@AAA	@.		\\
	0 @>>> CE(K) @>>>	H @>>>	 M	@>>>	0	\\
	@. @A{\subseteq}AA				@AAA	@|		\\
	0  @>>> K 	@>{\varepsilon}>>	P	@>{\pi}>> 	M	@>>>	0
\end{CD}$$
where $CE(K)$ denotes the cotorsion envelope of $K$. 
Since $\PGF^\perp$ contains all flat modules \chge{and $G$ is flat, we have $CE(K)\in\PGF^\perp$}. By $(3)$, the middle row splits whence the cotorsion envelope $K\hookrightarrow CE(K)$ factorizes through $\varepsilon$. In particular, \chge{$\varepsilon$~is a pure monomorphism and} $\pi$ is a pure epimorphism. 

Now, we form a special $\PGF^\perp$-preenvelope of $M$: $0 \to M\to F\to N\to 0$. Then $F\in\PGF^\perp$ and, at the same time, $\Ext_R^1(F,C) = 0$ for any cotorsion module $C\in\PGF^\perp$. By the preceding paragraph, $F$ is a pure-epimorphic image of a PGF-module. In particular, any special $\PGF$-precover $\rho\colon Q \to F$ is a pure epimorphism. Since $Q\in\PGF^\perp\cap\PGF$, it is projective, and so $F$ is flat.

\chge{The implication $(4)\Longrightarrow (1)$ follows at once from~\cite[Lemma 2.4]{B}.}

\smallskip

Finally, the inclusion $\mathcal{GF}\cap\PGF^\perp \supseteq \FL$ is a consequence of Theorem~\ref{t:pgf}, and the equivalent condition $(4)$ yields $\mathcal{GF}\cap\PGF^\perp \subseteq \FL$.

%
%
%
\end{proof}



The following corollary provides a generalization of \cite[Section 3]{G2}.

\begin{cor} \label{c:GF} There is a hereditary cotorsion pair $\mathfrak{GF} = (\mathcal{GF},\EC\cap\PGF^\perp)$ generated by a set of modules of cardinality at most $|R|+\aleph_0$. The kernel of $\mathfrak{GF}$ equals $\FL\cap\EC$. The class $\mathcal{GF}$ is closed under direct limits, and so it is a covering class \chge{(in the sense of~\cite[Definition 5.5]{GT})}. 
\end{cor}

\begin{proof} Let us denote $\mathcal B = \EC\cap\PGF^\perp$. We have $\mathcal{GF} = {}^\perp\mathcal B$ by Theorem~\ref{t:GF}(3). On the other hand, $\mathcal B = \mathcal{GF}^\perp$ since $\PGF\cup\FL\subseteq\mathcal{GF}$. So $\mathfrak{GF}$ is a cotorsion pair. In fact, it is the supremum of $(\FL,\EC)$ and $\mathfrak{PGF}$ in the big lattice of cotorsion pairs, using the convention from \cite[Chapter 12]{GT}. As such, it is generated by the union of two sets generating these cotorsion pairs, which means by a representative set of Gorenstein flat modules of cardinality at most $\nu = |R|+\aleph_0$. The description of the kernel $\mathcal{GF}\cap\mathcal{GF}^\perp$ of $\mathfrak{GF}$ stems frow Theorem~\ref{t:GF}(4).

In particular, we have shown that the class $\mathcal{GF}$ is closed under extensions, and so it is closed under direct limits as well by \cite[Lemma 3.1]{YL}. Finally, $\mathcal{GF}$ is covering by the well-known result due to Enochs, see e.g.\,\cite[Corollary 5.32]{GT}.
\end{proof}

\chge{Now we discuss the interpretation of our results from the perspective of stable model structures and the corresponding homotopy categories. To that end, recall that a triple $\mathcal H = (\mathcal Q,\mathcal W,\mathcal R)$ of classes of modules is called a \emph{Hovey triple} if $(\mathcal Q\cap \mathcal W,\mathcal R)$ and $(\mathcal Q,\mathcal W\cap \mathcal R)$ are complete cotorsion pairs and the class $\mathcal W$ is thick. Such triples were introduced in~\cite{Hov} and their basic properties are summarized for instance in~\cite[\S1]{Becker}. In particular, if the two cotorsion pairs associated with a Hovey triple are hereditary, it defines a stable model category \cite{Hov-book}, whose homotopy category is simply the stable category $\mathcal T$ of the Frobenius exact category $\mathcal Q\cap\mathcal R$ modulo the class $\mathcal Q\cap\mathcal W\cap\mathcal R$ of its projective-injective objects. It is well-known that $\mathcal T$ is a~triangulated category and it encodes the corresponding relative homological algebra on $\ModR$.

 The results of this section provide us, over any ring, with two previously unknown Hovey triples: $(\PGF, \PGF^\perp,\ModR)$ and $(\mathcal{GF},\PGF^\perp,\mathcal{EC})$. Since the middle classes of the triples coincide, the two stable model structures on $\ModR$ are Quillen equivalent in the sense of~\cite{Hov-book}. In particular, their homotopy categories are equivalent triangulated categories.

As mentioned at the beginning of the section, the class $\PGF$ contains the class of Gorenstein AC-projective modules from~\cite{BGH} and, if $R$ is left coherent, the two classes coincide. From the point of model structures, the Gorenstein AC-projective model structure from \cite[Theorem 8.5]{BGH} is a Bousfield localization of the one given by $(\PGF, \PGF^\perp,\ModR)$ (cf.\ \cite[\S\S1.4 and 1.5]{Becker}), and the localization is trivial if $R$ is left coherent. If $R$ is not left coherent, our model structures are potentially finer (if the Bousfield localization is proper). Unfortunately, we are presently not aware of any particular example of a non-coherent ring where the two model structures differ.

\smallskip
}

We also do not know whether $\mathcal{GF}$ is always closed under taking pure-epimorphic images, equivalently, whether $\mathcal{GF}$ is precisely the class of pure-epimorphic images of PGF-modules. Neither we know whether the implication $(2)\Longrightarrow (3)$ in the following result can be reversed, see also \cite[Problem 4.12]{CFH}.

\begin{prop} \label{p:lcoh} Consider the following conditions:
\begin{enumerate}
	\item $\mathcal{GF}$ is a definable class of modules.
	\item $\mathcal{GF}$ is closed under products.
	\item $R$ is left coherent.
	\item $\mathcal{GF}$ is closed under pure-epimorphic images and pure submodules.
\end{enumerate}
Then $(1)\Longleftrightarrow (2)\Longrightarrow (3) \Longrightarrow (4)$.
\end{prop}

\begin{proof} $(1)\Longrightarrow (2)$ is trivial. The converse will follow once we prove $(2)\Longrightarrow (4)$.

\smallskip

$(2)\Longrightarrow (3)$. Let $F$ be a flat module and $\kappa$ a cardinal. Then $F^\kappa\in\mathcal{GF}$ by $(2)$. Since $F^\kappa$ belongs also to $\PGF^\perp$, it is flat by Theorem~\ref{t:GF}(4). It follows that $R$ is left coherent.

\smallskip

The implication $(3)\Longrightarrow (4)$ is well-known, \cite[Theorems 3.6 and 2.6]{Ho}.
\end{proof}

\begin{exm} \cite[Section 4, Ex.(1)]{EIY} The ring $$R= \left(\begin{matrix}\Q&0&0 \\ \Q&\Q&0 \\ \Rl & \Rl & \Q \end{matrix}\right)\Bigl/\left(\begin{matrix} 0&0&0 \\ 0&0&0 \\ \Rl & 0 & 0\end{matrix}\right)$$ is (left and) right perfect and not left coherent. Moreover, the class $\mathcal{GF}$ coincides with the class of all projective modules. It follows that the implication $(4) \Longrightarrow (3)$ in Proposition~\ref{p:lcoh} does not hold for this $R$.
\end{exm}

\section{Gorenstein injective modules}
\label{sec:GI}

In the whole section, $\mathcal I_0$ denotes the class of all injective modules. Recall that a~complex $I^\bullet$ of injective modules is \emph{totally acyclic} if it is exact and it remains exact after the application of $\Hom_R(E,-)$ with $E$ arbitrary injective. We write $\Chtac{\mathcal I_0}$ for the class of all totally acyclic complexes of injective modules. A~module $M$ is called \emph{Gorenstein injective} if it is a syzygy in a totally acyclic complex of injective modules. We denote by $\mathcal{GI}$ the class of all Gorenstein injective modules. In this section, we show that, over any ring, the class $\mathcal{GI}$ forms the right-hand class of a~perfect hereditary cotorsion pair $\mathfrak{GI} = (\mathcal W,\mathcal{GI})$ with $\mathcal W$ thick.

\smallskip

We start with \chge{general observations. 

\begin{lem} \label{l:thick-to-lim}
Let $\mathcal C$ be a class of modules which is thick and closed under filtrations (i.e.\ $M\in\mathcal C$ whenever $M$ has a filtration with consecutive factors in $\mathcal C$). Then $\mathcal C$ is closed under direct limits.
\end{lem}

\begin{proof}
The proof of \cite[Proposition 3.1]{G1} applies. Although there $\mathcal C$ is assumed to be a thick left-hand side of a cotorsion pair, the proof only uses the fact that $\mathcal C$ is closed under transfinite extensions by the Eklof Lemma \cite[Lemma 6.2]{GT}.
\end{proof}
}

\begin{lem} \label{l:thickfilt} Let $\mathcal C$ be a class of modules \st $\mathcal A = {}^\perp\mathcal C$ is thick, and let $\kappa\geq|R|$ be an infinite cardinal. Assume that $M$ is an almost $(\mathcal C,\kappa^+)$-projective module.

Then $\mathcal A$ is closed under direct limits and $M$ possesses a filtration $\mathcal M = (M_\alpha \mid \alpha\leq\sigma)$ where $M_{\alpha+1}/M_\alpha$ is a $\kappa$-presented module from $\mathcal A$ for each $\alpha<\sigma$.
\end{lem}

\begin{proof} \chge{First, note that $\mathcal A$ is closed under direct limits by Lemma~\ref{l:thick-to-lim}; in particular, $M\in\mathcal A$.} Let $\mathcal S$ be a system witnessing that $M$ is almost $(\mathcal C,\kappa^+)$-projective. Since $\kappa \geq |R|+\aleph_0$, we can w.l.o.g.\ assume that $\mathcal S$ consists of submodules of $M$ and inclusions (use Construction~\ref{conmerge}). Suppose that the module $M$ is $\mu$-presented. We proceed by transfinite induction on $\mu$. If $\mu\leq\kappa$, there is nothing to prove, so assume $\mu>\kappa$.

Let $\mu$ be a regular cardinal. Using the notation from Definition~\ref{d:inducesys}, we see that the system $\mathcal S^\mu$ consists of $<\mu$-presented submodules of $M$ which belong to $\mathcal A$. Since $M\in\mathcal A$ and $\mathcal A$ is thick, we can choose from $\mathcal S^\mu$ a filtration $\mathcal T = (N_\alpha \mid \alpha<\mu)$ of $M$ with consecutive factors $<\mu$-presented and belonging to $\mathcal A$.

By the definition of $\mathcal S^\mu$ and the thickness of $\mathcal A$, $\mathcal T$ consists of $(\mathcal C,\kappa^+)$-projective modules (use Observation~\ref{o:observ}). Thus we can use Construction~\ref{concoker} for each $\alpha<\mu$ to obtain a system $\mathcal K_\alpha$ witnessing that $N_{\alpha+1}/N_\alpha$ is almost $(\mathcal C,\kappa^+)$-projective. By the inductive hypothesis, $N_{\alpha+1}/N_\alpha$ possesses a filtration with consecutive factors $\kappa$-presented modules from $\mathcal A$ for each $\alpha<\mu$. We use these filtrations to refine $\mathcal T$ into the desired filtration~$\mathcal M$.

If $\mu$ is singular, we see that $M$ is almost $(\mathcal C,\lambda)$-projective for all regular $\kappa<\lambda<\mu$ using that $\mathcal A$ is closed under direct limits. The inductive hypothesis and \cite[Theorem 7.29]{GT} \chge{(singular compactness)} give us the filtration $\mathcal M$.
\end{proof}

\begin{lem} \label{l:injfilt} Put $\lambda = |R|+\aleph_0$ and let $\kappa$ be an infinite cardinal such that $\kappa = \kappa^\lambda$. Assume that $A\in \mathcal A = {}^\perp\mathcal C$ is a pure-injective module. Then $A$ is almost $(\mathcal C,\kappa^+)$-projective provided that $\mathcal A$ is closed under direct limits (of monomorphisms).
\end{lem}

\begin{proof} We use \cite[Lemma 10.5]{GT} to obtain, for any $X\subseteq A$ of cardinality $\leq\kappa$, a~$\kappa$-presented submodule $C$ of $A$ \st $X\subseteq C$ with the property that each system of cardinality $\leq\lambda$ consisting of $R$-linear equations with parameters from $C$ has a~solution in $C$ provided that it has a~solution in $A$. In particular, $C$ is pure in $A$, and it follows from \cite[Theorem V.1.2]{EM} that $C$ is pure-injective (here, we use that $A$ is pure-injective). Thus $C$ is a direct summand in $A$ which yields $C\in\mathcal A$.

By the preceding paragraph, we know that $A$ is the directed union of a system consisting of $\kappa$-presented direct summands of $A$. If $\mathcal A$ is closed under directed unions, we can enlarge this system to a one witnessing that $A$ is almost $(\mathcal C,\kappa^+)$-projective.
\end{proof}


\begin{lem} \label{l:GIthick} The class ${}^\perp\mathcal{GI}$ is thick. Subsequently, for any infinite cardinal $\kappa$ \st $\kappa ^{|R|+\aleph_0} = \kappa$, each pure-injective module in ${}^\perp\mathcal{GI}$ possesses a filtration with consecutive factors $\kappa$-presented modules from ${}^\perp\mathcal{GI}$. In particular, this is the case of any injective module in $\ModR$.
\end{lem}

\begin{proof} The first, well known, part is straightforward from the definition of a Gorenstein injective module. The second part follows from the two preceding lemmas.
\end{proof}

Now we can prove that the class of totally acyclic complexes of injective modules forms the right-hand side of a~cotorsion pair generated by a~set in the category $\ChR$ of complexes of right $R$-modules.

\begin{prop} \label{p:GIcomplex} Let $\lambda = |R| + \aleph_0$ and $\kappa$ be the least infinite cardinal \st $\kappa^\lambda = \kappa$. There is a cotorsion pair $\mathfrak B = (\mathcal B,\Chtac{\mathcal I_0})$ in $\ChR$ generated by a~set of $\kappa$-presented complexes, hence complete. Moreover, each $B^\bullet\in\mathcal{B}$ has a filtration with consecutive factors $\kappa$-presented complexes from $\mathcal B$.
\end{prop}

\begin{proof} The class $\Chac{\mathcal I_0}$ of all exact complexes of injective modules forms the right-hand class of a cotorsion pair generated by a set $S$ consisting of $\kappa$-presented (in fact, even $\lambda$-presented) complexes by \cite[Proposition 4.6]{G3}. Denote by $T^\prime$ a representative set of all $\kappa$-presented modules from ${}^\perp\mathcal{GI}$. Set $T = \{S_n(M) \mid M\in T^\prime, n\in\Z\}$ where $S_n(M)$ denotes the stalk complex concentrated in degree $n$. We claim that $(S\cup T)^\perp = \Chtac{\mathcal I_0}$.

Indeed, an exact complex $I^\bullet$ of injective modules is totally acyclic \iff each cocycle module $Z^n(I^\bullet), n\in\Z$, is Gorenstein injective, or equivalently in~$\mathcal I_0^\perp$. Since $\Ext_{\ChR}^1(S_n(M),I^\bullet)\cong \Ext_R^1(M,Z^n(I^\bullet))$, we get the desired conclusion by the Eklof Lemma and Lemma~\ref{l:GIthick}. The completeness of the cotorsion pair follows, for instance, by \cite[Proposition 2.12]{SaoSt}. The final claim then by \cite[Theorem 7.13]{GT} which holds in any finitely accessible Grothendieck category.
\end{proof}

It is easy to show that \chge{ the cotorsion pair $\mathfrak B$ is hereditary and its kernel $\B\cap\Chtac{\mathcal I_0}$} coincides with the class of all (categorically) injective complexes. As a~consequence, it follows that $\mathcal B$ is also a~thick class. By \cite[Theorem 1.2]{G4}, we can obtain a Hovey triple $(\mathcal B, \mathcal Y, \mathcal G)$ in $\ChR$, where $\mathcal G$ denotes the class of all (categorically) Gorenstein injective complexes, using the following theorem for $\ChR$ instead of $\ModR$.
\chge{Recall that a cotorsion pair $\mathfrak C=(\A,\B)$ is \emph{perfect} if $\A$ is a covering class and $\B$ is enveloping (see \cite[Definition 5.26]{GT} for details).}

\begin{thm} \label{t:GI} Let $\mathcal W = {}^\perp\mathcal{GI}$. The pair $\mathfrak{GI} = (\mathcal W,\mathcal{GI})$ is a hereditary, perfect cotorsion pair generated by a set of $\kappa$-presented modules where $\kappa$ is the least infinite cardinal \st $\kappa^{|R|+\aleph_0} = \kappa$. In particular, every module has a $\mathcal{GI}$-envelope.
\end{thm}

\begin{proof} We are going to give two proofs. The first one using Proposition~\ref{p:GIcomplex} and the second one entirely in the category $\ModR$.

For the first, set $\mathfrak C = (\mathcal W,\mathcal W^\perp)$ and let $M\in\mathcal W$ be arbitrary. Then $S_0(M)\in \mathcal B = {}^\perp \Chtac{\mathcal I_0}$. Thus $S_0(M)$ possesses a filtration with consecutive factors $\kappa$-presented complexes in $\mathcal B$ by Proposition~\ref{p:GIcomplex}. However, all these complexes are concentrated in degree $0$, hence they induce a filtration of $M$ with consecutive factors $\kappa$-presented modules in $\mathcal W$. Since $M$ was arbitrary, it follows from the Eklof Lemma that $\mathfrak{C}$ is generated by a representative set of $\kappa$-presented modules from $\mathcal W$, hence $\mathfrak{C}$ is complete.

Since $\mathcal W$ is thick by Lemma~\ref{l:GIthick}, the cotorsion pair $\mathfrak C$ is hereditary. Moreover, the kernel of $\mathfrak C$ equals $\mathcal I_0$: indeed, let $M\in\mathcal W^\perp\cap\mathcal W$; then the injective envelope $E(M)$ of $M$ belongs to $\mathcal W$, and so $E(M)/M\in\mathcal W$ too by the thickness of $\mathcal W$, which implies that $M$ splits in $E(M)$.

We want to show that $\mathfrak C = \mathfrak{GI}$. The only thing we need to check is that $\mathcal W^\perp\subseteq \mathcal{GI}$. So suppose $G\in\mathcal W^\perp$ is arbitrary. Iteratively forming special $\mathcal W$-precovers of~$G$, we obtain a long exact sequence

$$\dotsb\longrightarrow I^n \longrightarrow I^{n+1} \longrightarrow \dotsb \longrightarrow I^{-1} \longrightarrow G \longrightarrow 0$$
where $I^n$ is injective for each $n<0$ and all syzygies belong to $\mathcal W^\perp$. At the same time, all cosyzygies in any injective coresolution of $G$ belong to $\mathcal W^\perp$ since $\mathfrak C$ is hereditary. We conclude that $G$ is a syzygy in a totally acyclic complex of injective modules, \chge{or in other words} $G\in\mathcal{GI}$.

Finally, the cotorsion pair $\mathfrak{GI}$ is perfect by \cite[Corollary 5.32]{GT} since $\mathcal W$ is closed under direct limits by Lemma~\ref{l:thick-to-lim}. In particular, every module has a $\mathcal{GI}$-envelope.

\smallskip

Alternatively, we can argue as follows. Let $\mathcal S$ be a representative class of all $\kappa$-presented modules from ${}^\perp\mathcal{GI}$ and $\mathfrak C = (\mathcal X, \mathcal S^\perp)$ be the cotorsion pair generated by $\mathcal S$. By Eklof Lemma and Lemma~\ref{l:GIthick}, all injective modules belong to $\mathcal X$. Our goal is to show that $\mathcal X$ is thick.

Since $\kappa\geq |R|$ and ${}^\perp\mathcal{GI}$ is closed under kernels of epimorphisms, each module from $\mathcal S$ has a syzygy in $\mathcal S$ whence $\mathfrak C$ is hereditary. From \cite[Theorem 7.13]{GT}, we know that each module in $\mathcal X$ possesses a filtration with consecutive factors (isomoprhic to elements) in $\mathcal S$. Using Hill Lemma \cite[Theorem 7.10\,(H4)]{GT}, we further get, for each $M\in\mathcal X$, a system $\mathcal M$ of $\kappa$-presented submodules of $M$ witnessing that $M$ is almost $(\mathcal{GI},\kappa^+)$-projective. In fact, since $M\in {}^\perp\mathcal{GI}$ and ${}^\perp\mathcal{GI}$ is thick, $\mathcal M$ witnesses even $(\mathcal{GI},\kappa^+)$-projectivity of $M$. Having two modules $M_1, M_2\in\mathcal X$ with $M_1\subseteq M_2$ and their respective systems $\mathcal M_1, \mathcal M_2$, we use Construction~\ref{concoker} to show that $M_2/M_1$ is almost $(\mathcal{GI},\kappa^+)$-projective. Consequently, $M_2/M_1$ belongs to $\mathcal X$ by Lemma~\ref{l:thickfilt} and Eklof Lemma.

We have proved that $\mathcal X$ is thick and contains all injective modules. As before, we observe that $\mathcal X \cap \mathcal S^\perp = \mathcal I_0$ and finally show that $\mathfrak C = \mathfrak {GI}$ by the same argument as in the first proof.
\end{proof}

As an immediate consequence, we obtain the Hovey triple $(\ModR,\mathcal W,\mathcal{GI})$ in $\ModR$.
\chge{The class of Gorenstein injectives contains the class of Gorenstein AC-injectives, as defined in~\cite[\S5]{BGH}. Hence, similarly to the previous section, our Hovey triple refines the Gorenstein AC-injective model structure from \cite[Theorem 5.5]{BGH}, in that the model structure from \cite{BGH} is a Bousfield localization of ours.}

Unfortunately, similarly as in the previous section, we do not know whether $\mathcal W$ is, in general, closed under pure submodules or, equivalently, pure-epimorphic images.

\section{Singular step --- set theoretical part}
\label{sec:singset}

In the rest of the paper, we denote by $0$ the empty set and by $\nu$ a fixed infinite cardinal. The qualification \emph{countable} is intended as \emph{having cardinality $<\aleph_1$}.

\smallskip

Let $(I,\leq)$ be a directed (i.e. upward directed) poset with $\kappa = |I|$ singular and $\nu<\kappa$. Set $\mu = \cf(\kappa)$. For each successor cardinal $\lambda$ such that $\nu<\lambda<\kappa$, we fix a~set $\mathcal I_\lambda\subseteq [I]^{<\lambda}$ and put $\mathcal I = \bigcup_{\lambda}\mathcal I_\lambda$.
We assume that for each successor cardinal $\lambda$, where $\nu<\lambda<\kappa$, the set $\mathcal I_\lambda$ has the following properties:
\begin{enumerate}
\setcounter{enumi}{-1}
	\item $0\in\mathcal I_\lambda$;
	\item $\mathcal I_\lambda$ consists of directed subposets of $(I,\leq)$;
	\item for each $A\in [I]^{<\lambda}$ there exists $B\in\mathcal I_\lambda$ \st $A\subseteq B$;
	\item if $\mathcal C\subseteq \mathcal I_\lambda$ is a $\subseteq$-chain with $|\mathcal C|<\lambda$, then $\bigcup\mathcal C\in\mathcal I_\lambda$.
\end{enumerate}

Let us denote by $\mathcal W$ the set of all pairs $(A,B)$ of directed subposets of $(I,\leq)$ such that $A\subseteq B\in [I]^{<\kappa}$. For elements $(A,B),(C,D)\in\mathcal W$, we shall write $(A,B)\subseteq (C,D)$ \iff $A\subseteq C\,\&\,B\subseteq D$. This makes $\mathcal W$ into a poset.

\smallskip

Put $\mathcal W_0 = \{(0,A)\mid (0,A)\in\mathcal W\}$. In particular $(0,0)\in\mathcal W_0$. 
For $V= (A,B)\in\mathcal W$, we define $|V|$ as $|B|$. For an element $V\in\mathcal W$ and a cardinal $\lambda$, we use the notation $[V]^{<\lambda}$ to denote the set of all $W\in\mathcal W$ with $W\subseteq V$ and $|W|<\lambda$. Finally, for $\mathcal S\subseteq \mathcal W$, the union $\bigcup\mathcal S$ is computed component-wise, i.e. if $\mathcal S = \{(A_j,B_j)\mid j\in J\}$, then $\bigcup\mathcal S = (\bigcup _{j\in J}A_j,\bigcup_{j\in J} B_j)$.

\smallskip

In Section~\ref{sec:singmod}, we are going to apply these tools in the following context: The poset $I$ will index a directed system $\mathcal S$ of (countably presented) modules whose direct limit is a $\kappa$-presented module of our interest. The elements in $\mathcal W_0$ will correspond to direct limits of `small' directed subsystems of $\mathcal S$, while the elements from $\mathcal W$ will correspond to cokernels of canonical morphisms between these. Finally, the sets $\{(0,A)\mid A\in\mathcal I_\lambda\}$ will encode systems witnessing almost $(\mathcal D,\lambda)$-projectivity (for a particular class $\mathcal D$), and a binary relation $\preceq$ satisfying the axioms below will be used to capture the $\mathcal D$-injectivity of canonical morphisms between the modules corresponding to elements from $\mathcal W$.

\begin{defn} \label{d:axioms} Consider a binary relation $\preceq$ on $\mathcal W$ satisfying the following axioms.
\begin{enumerate}
 \item $\preceq$ is a partial order of $\mathcal W$ with the smallest element $(0,0)$.
\smallskip
 \item $V\preceq W \Rightarrow V\subseteq W$.
\smallskip 
 \item ($V\subseteq W\subseteq X\;\&\; V\preceq X) \Rightarrow V\preceq W$.
\smallskip 
 \item For every $V=(A,B)\in\mathcal I^2$ with $(0,A)\preceq (0,B)$ and any successor cardinal $\lambda$ \st $\nu<\lambda\leq |B|$, there is a system $V(\lambda)\subseteq[V]^{<\lambda}$ such that:
\begin{enumerate}
  \item $(\forall (C,D)\in V(\lambda))\, (C,D)\preceq V,(0,C)\preceq (0,A), (0,D)\preceq (0,B)$;
  \item $\bigcup V(\lambda) = V$;
  \item $V(\lambda)$ is upwards directed;
  \item if $\mathcal C\subseteq V(\lambda)$ is a chain and $|\mathcal C|<\lambda$, then $\bigcup\mathcal C \in V(\lambda)$.
\end{enumerate}
\smallskip
 \item For any $(A,B),(C,D)\in\mathcal W$ with $(A,B)\subseteq (C,D)$ we have $(A,B)\preceq (C,D)$, if $(0,A)\preceq (0,D), (0,B)\preceq (0,D), (0,C)\preceq (0,D)$ and $(A,C)\preceq (B,D)$.
 \smallskip
 \item Let $(A_0,B_0)\preceq (A_1,B_1)\preceq (A_2,B_2)\preceq \dotsb$. Put $(A_\omega, B_\omega) = \bigcup _{k\in\omega} (A_k,B_k)$. If $(0,A_k)\preceq (0,B_\omega)$ and $B_k\in\mathcal I$ for all $k\in\omega$, then $(0,A_\omega)\preceq (0,B_\omega)$.
\end{enumerate}
\end{defn}

First, we need an auxiliary lemma which does not use the last two axioms. Moreover, Axiom $(4)$ is used only for the case $A=0$. The main source of inspiration for its proof comes from \cite[Proposition IV.3.4]{EM}.

\begin{lem} \label{l:game} For any $X\in\mathcal W_0$ of cardinality at least $\nu$, there exists $N\in\mathcal W_0\cap\mathcal I^2_{|X|^+}$ such that $X\subseteq N$ and with the property that $N\preceq Y$ whenever $Y\in\mathcal W_0$, $|Y| = |N|$ and $N\subseteq Y$.
\end{lem}

\begin{proof} Put $\lambda = |X|$. For $N\in\mathcal W_0\cap\mathcal I^2_{\lambda^+}$, we define the $N$-Shelah game. It is played in turns by two players. Player I starts and chooses successively elements $X_0, X_1,\dotsc$ from $\mathcal W_0$ of cardinality at most $\lambda$. Player II, on each $X_n$, replies with some $N_n\in\mathcal W_0\cap\mathcal I^2_{\lambda^+}$. 
At most $\omega$ turns are played; after the first $n+1$ turns, we will have the following sequence:
$$X_0, N_0, X_1, N_1,  \dotsc , X_n, N_n.$$
\noindent
Player II wins, if he manages to play, for each $n\in\omega$, so that $X_n\subseteq N_n$ and $N_{n-1}\preceq N_n$ where we put $N_{-1} = N$. Otherwise, Player I immediately wins. Let $\mathcal S$ denote the set of all $N\in\mathcal W_0\cap\mathcal I^2_{\lambda^+}$ for which Player I possesses no winning strategy in $N$-Shelah game. We show that $(0,0)\in\mathcal S$.

\smallskip

First, for each $K\in\mathcal W_0\cap\mathcal I^2_{\lambda^{++}}$, we fix an $\subseteq$-increasing chain $(K^\alpha\in\mathcal W_0 \mid |K^\alpha|\leq\lambda, \alpha<\lambda^+)$ \st $\bigcup_{\alpha<\lambda^+} K^\alpha = K$.
Let $s$ be a strategy for Player I in $(0,0)$-Shelah game, i.e. a function that gives the first move $X_0$, and it decides what the answer should be to the play by Player II; so $X_n = s(N_0, N_1, \dots , N_{n-1})$ for $n>0$. We want to beat the strategy $s$. Using the properties of $\mathcal I$ and $\preceq$, we inductively construct increasing sequences $(M_\alpha\in \mathcal W_0\cap\mathcal I^2_{\lambda^+}\mid \alpha <\lambda^+)$ and $(K_\alpha \in \mathcal W_0\cap \mathcal I^2_{\lambda^{++}} \mid \alpha<\lambda^+)$ in such a way that:
\begin{enumerate}
\setcounter{enumi}{-1}
\item $X_0\subseteq M_0$; 
\item $M_\alpha = \bigcup _{\beta <\alpha} M_\beta$ for $\alpha <\lambda^+$ limit;
\item $M_\alpha\subseteq K_\alpha$ for each $\alpha<\lambda^+$;
\item $M_{\alpha + 1}\supsetneq M_\alpha\cup \bigcup _{\beta\leq\alpha}K^\alpha_\beta$ for each $\alpha<\lambda^+$;
\item for each $\alpha <\lambda^+$, $s(M_{\alpha _0}, M_{\alpha _1}, \dotsc , M_{\alpha _n})\subseteq M_{\alpha + 1}$, whenever $n\in\omega$, $\alpha _n\leq \alpha$ and $M_{\alpha _0}\preceq M_{\alpha _1}\preceq \dotsb \preceq M_{\alpha _n}$ is played by Player II according to the rools (against the strategy $s$).
\end{enumerate}

Put $M = \bigcup _{\alpha <\lambda^+} M_\alpha$. We have $|M| = \lambda^+$ and $M = \bigcup_{\alpha<\lambda^+} K_\alpha\in\mathcal I^2$ by $(3)$. Considering the system $M(\lambda^+)$ given by Axiom $(4)$ from Definition~\ref{d:axioms}, it is easy to see that the set $\{\beta <~\lambda^+ \mid M_\beta\in M(\lambda^+)\}$ is unbounded in $\lambda^+$. Player II is going to beat the strategy~$s$, if he chooses the elements $N_n$ as the appropriate $M_\beta$ for $\beta$ from this unbounded set.

\smallskip

Finally, it is enough to notice that, for $X_0 = X$, it is possible to play in the $(0,0)$-Shelah game such $N_0 = N$ that $N\in\mathcal S$. If not, Player I would have possessed a winning strategy in the $(0,0)$-Shelah game, a contradiction. This $N$ is the one we were looking for.
\end{proof}

The main result of this section follows. Its proof employs a nontrivial enhancement of techniques coming from \cite[Theorem IV.3.3]{EM}. Note that the slightly unusual (re)definition of sets $\mathcal B^n_\alpha$ condenses a back-and-forth construction whose explication would just make the proof look even more technical.

\begin{thm} \label{t:setchain} There is an increasing continuous $\preceq$-chain $\{(0,C_\alpha)\in\mathcal I^2 \mid \alpha<\mu\}$ with $\bigcup_{\alpha<\mu}C_\alpha =~I$.
\end{thm}

\begin{proof} Fix a strictly increasing continuous chain $(\nu_\alpha \mid \alpha<\mu)$ of infinite cardinals cofinal in~$\kappa$ such that $\nu_0>\mu+\nu$. For $n\in\omega$, we recursively define strictly increasing chains $(V^n_\alpha\in\mathcal W_0 \mid \alpha<\mu)$, where $V^n_\alpha = (0,A^n_\alpha)$, together with (arbitrarily fixed) enumerations $A^n_\alpha = \{a^n_{\alpha,\beta}\mid \beta < \nu_\alpha\}$ as follows.

Pick $V^0_\alpha\in\mathcal W_0\cap \mathcal I^2_{\nu_\alpha^+}$ of cardinality $\nu_\alpha$ with $\bigcup_{\beta<\alpha} V^0_\beta\subseteq V^0_\alpha$ and such that for any $Y\in\mathcal W_0$ of cardinality $\nu_\alpha$ with $V^0_\alpha\subseteq Y$, we have $V^0_\alpha\preceq Y$. This is possible by Lemma~\ref{l:game}. Moreover, we can assume that $\bigcup _{\alpha<\mu} A^0_\alpha = I$.

For $n = 1$, choose $V^1_\alpha$ from $V^0_{\alpha+1}(\nu_\alpha^+)$ arbitrarily so that $V^0_\alpha\subseteq V^1_\alpha$. 
Additionally, put $\mathcal B^1_\alpha = \{B \mid (0,B)\in V^0_{\alpha+1}(\nu_\alpha^+), A^1_\alpha\subseteq B\}$.

For $n>0$ even, we choose the $V^n_\alpha$ again from the set $\mathcal I^2_{\nu_\alpha^+}$, with $V^n_\alpha\supseteq \bigcup_{\beta<\alpha} V^n_\beta$ and such that for any $Y\in\mathcal W_0$ of cardinality $\nu_\alpha$ with $V^n_\alpha\subseteq Y$, we have $V^n_\alpha\preceq Y$. Furthermore, using the assumption $\nu_0>\mu$, we can demand that $$A^n_\alpha \supseteq \{a^{n-1}_{\gamma,\beta} \mid \gamma<\mu, \beta<\min\{\nu_\gamma, \nu_\alpha\}\}.\eqno(\dagger)$$

\smallskip

Let $n >1$ odd. By the construction, we have $A^{n-3}_{\alpha+1}, A^{n-1}_{\alpha+1}\in\mathcal I$ and $V^{n-3}_{\alpha+1}\preceq V^{n-1}_{\alpha + 1}$. We also assume that the sets $\mathcal B^i_\alpha$, $i<n$ odd, constructed in the previous odd steps are upwards directed, closed under unions of chains of cardinality $\leq\nu_\alpha$, and $\bigcup\mathcal B^i_\alpha = A^{i-1}_{\alpha+1}$.

Using Axiom $(4)$, we can pick arbitrary $(A^{n-2,n}_\alpha, A^n_\alpha)\in (A^{n-3}_{\alpha+1},A^{n-1}_{\alpha+1})(\nu_\alpha^+)$ such that $A^{n-2,n}_\alpha\in\mathcal B^{n-2}_\alpha$ and $A^{n-1}_\alpha\subseteq A^n_\alpha$. By the definition of $\mathcal B^{n-2}_\alpha$ (see below), we have an induced chain $(0,A^{1,n}_\alpha)\preceq (0,A^{3,n}_\alpha)\preceq\dotsb\preceq (0,A^{n-2,n}_\alpha)$ satisfying, for each odd $i<n$, $A^{i,n}_\alpha\in\mathcal B^{i}_\alpha$ and $(A^{i,n}_\alpha,A^{i+2,n}_\alpha)\in (A^{i-1}_{\alpha+1},A^{i+1}_{\alpha+1})(\nu_\alpha^+)$ where we put $A^{n,n}_\alpha = A^n_\alpha$. Set $V^n_\alpha = (0, A^n_\alpha)$.

For each $i<n$ odd, we gradually replace the sets $\mathcal B^i_\alpha$ by their subsets $\{B\in\mathcal B^i_\alpha \mid A^{i,n}_\alpha\subseteq B\}$. (Redefining $\mathcal B^1_\alpha$ alters the set $\mathcal B^3_\alpha$ which we replace by the subset as above, hereby we modify $\mathcal B^5_\alpha$, and so on.) This does not harm the properties of~$\mathcal B^i_\alpha$ mentioned above. At the same time, it guarantees that we will get $(0,A^{i,n}_\alpha)\preceq (0,A^{i,n+2}_\alpha)$ for each $i\in\{1,3,5,\dots,n\}$ in the next odd step if we define
$$\mathcal B^n_\alpha = \{B\mid (\exists A\in \mathcal B^{n-2}_\alpha) (A,B)\in (A^{n-3}_{\alpha+1},A^{n-1}_{\alpha+1})(\nu_\alpha^+)\;\&\;A^n_\alpha\subseteq B\}.$$
Notice that $\mathcal B^n_\alpha$ has the properties required in the next odd step too.

\medskip

We claim that $\mathcal S = (\bigcup_{j<\omega}V^j_\alpha \mid \alpha<\mu)$ is the closed $\preceq$-chain we have been looking for. First of all, we know that $\bigcup_{j<\omega}V^j_\alpha\in\mathcal I^2_{\nu_\alpha^+}$ for all $\alpha<\mu$, using the even steps. Further, it immediately follows from the property $(\dagger)$ that $\mathcal S$ is continuous. It remains to prove that $\mathcal S$ is a $\preceq$-chain.

Now fix $\alpha<\mu$. For each $k\in\omega$, put $B_k = A^{2k}_{\alpha + 1}$, $A_k = \bigcup_{k\leq j<\omega} A^{2k+1,2j+1}_\alpha$, $W_k = (0,B_k)$ and $V_k = (0,A_k)$. By the construction, we have the chain $V_0\preceq V_1\preceq V_2\preceq\dotsb$: indeed, for each $k\in\omega$, $V_k\preceq W_k$ by Axiom (4), furter $W_k\preceq W_{k+1}$ by the even-step incorporation of Lemma~\ref{l:game}, finally $V_k\preceq V_{k+1}$ by Axiom (3) since $V_k\subseteq V_{k+1}\subseteq W_{k+1}$. Moreover, we have 
$(A_k,A_{k+1})\in (B_k,B_{k+1})(\nu_\alpha^+)$ for each $k<\omega$, and $\bigcup_{k<\omega} A_k = \bigcup _{j<\omega} A^j_\alpha$.

We want to use Axiom $(6)$ for the setting $(A_k,B_k)$. First notice that, for each $k<\omega$, the set $B_k$ belongs to $\mathcal I$ and we have $(0,A_k)\preceq (0,B_k)$ since $(A_k,A_{k+1})\in (B_k,B_{k+1})(\nu_\alpha^+)$, and $(0,B_k)\preceq (0,B_\omega)$ by the choice of $(0,B_k) = V^{2k}_{\alpha+1}$. Thus $(0,A_k)\preceq (0,B_\omega)$ for each $k<\omega$.

The relation $(A_k,B_k)\preceq (A_{k+1},B_{k+1})$ now follows immediately from Axiom~$(5)$ since $(A_k,A_{k+1})\preceq (B_k,B_{k+1})$ holds. Hence the hypotheses of Axiom $(6)$ are satisfied, and we can conclude that $\bigcup_{j<\omega}V^j_\alpha\preceq \bigcup_{j<\omega}V^j_{\alpha+1}$ for each $\alpha<\mu$.
\end{proof}

\begin{rem} Analyzing the proof in detail, we see that Axiom $(4)$ is needed only for pairs $(\lambda, |B|^+)$ of the form $(\nu_\alpha^+,\nu_\alpha^{++})$ or $(\nu_\alpha^+,\nu_{\alpha+1}^+)$ for $\alpha<\mu$.
\end{rem}

We finish this section with a lite version of the theorem above.

\begin{prop} \label{p:lite} Assume that, for all $A\in\mathcal I$, we have $(0,A)\preceq (0,B)$ whenever $(A,B)\in\mathcal W$. Then we obtain the same conclusion as in Theorem~\ref{t:setchain} without using Axioms $(5)$ and $(6)$, and with Axiom $(4)$ used only for the case $A = 0$.
\end{prop}

\begin{proof} Follow the proof of Theorem~\ref{t:setchain}. In odd steps, instead of the subtle technical dance, do the same as in the even ones.
\end{proof}

\section{Singular step --- modules}
\label{sec:singmod}

\begin{lem} \label{l:singmod} Let $R$ be a ring with enough idempotents. Let $\kappa$ be a singular cardinal, $M$ be a $\kappa$-presented module and $\mathcal D$ a filter-closed class of modules. Assume that there is an infinite cardinal $\nu$ such that, for all successor cardinals $\nu<\lambda<\kappa$, there is a~system $\mathcal S_\lambda$ witnessing that $M$ is almost $(\mathcal D,\lambda)$-projective. Then $M\in {}^\perp\mathcal D$.

Furthermore, if $\mathcal S_\lambda$ witnesses even the $(\mathcal D,\lambda)$-projectivity of $M$ for each successor cardinal $\nu<\lambda<\kappa$, then the same conclusion holds regardless of whether $\mathcal D$ is filter-closed or not.
\end{lem}

\begin{proof} We fix a directed system $\mathcal S = (M_i,f_{ji}:M_i\to M_j \mid i\leq j\in I)$ with $\varinjlim \mathcal S = M$, and such that $(I,\leq)$ is a directed poset of cardinality $\kappa$ and $M_i$ is countably presented for all $i\in I$. 
Moreover, using Construction~\ref{conmerge} for $\gamma = 1$, we can w.l.o.g. assume that, for each $\lambda$, the system $\mathcal S_\lambda$ consists of direct limits of some directed subsystems of $\mathcal S$ of cardinality $<\lambda$ and canonical colimit factorization maps between them.
We define $\mathcal I_\lambda$ as the set of all the underlying directed posets of these subsystems of $\mathcal S$. Set $\mathcal I = 
\bigcup_{\nu<\cf(\lambda)=\lambda<\kappa} \mathcal I_\lambda$.

Recalling the previous section, we consider the functor $\Phi$ from the category $(\mathcal W,\subseteq)$ to $\ModR$ sending an element $(A,B)$ to the cokernel of the canonical colimit factorization map from $\varinjlim_{i\in A} M_i$ to $\varinjlim _{i\in B} M_i$, an inclusion $(0,A)\subseteq (0,B)$ to this colimit factorization map, and an inclusion $(A,B)\subseteq (C,D)$ to the map $f$ uniquely determined by the following commutative diagram:

$$\begin{CD}
	\Phi(0,C) 	@>>>	\Phi(0,D)	 @>>> 		\Phi(C,D)			@>>>	0	\\
	@AAA				@AAA	@A{f}AA		\\
	\Phi(0,A)  @>>>	 \Phi(0,B)  @>>>	 	\Phi(A,B)	@>>>	0.
\end{CD}$$

\medskip

For $V,W\in \mathcal W, V\subseteq W$, we define the order relation $\preceq$ by setting $V \preceq W \Longleftrightarrow \Phi(V)\to \Phi(W)$ is $\mathcal D$-injective. Notice that for $A\in\mathcal I$, one gets $\Phi(0,A)\in {}^\perp\mathcal D$. 
It follows from Lemma~\ref{l:eklof} that to show $M\in {}^\perp\mathcal D$ it is sufficient to verify the axioms from Definition~\ref{d:axioms} and apply Theorem~\ref{t:setchain}.

\smallskip

Axioms $(1)$ to $(3)$ easily hold. In the Axiom $(4)$ for $V = (A,B)$, the modules $\Phi(0,A),\Phi(0,B),\Phi(A,B)$ belong to ${}^\perp\mathcal D$.

Thus we can use \cite[Lemma 2.3]{SS} to find upward directed sets $\mathcal A$, $\mathcal B\subseteq [I]^{<\lambda}$ closed under unions of chains of length $<\lambda$ and consisting of directed subposets of $(I,\leq)$ \st $\bigcup\mathcal A = A, \bigcup\mathcal B = B$ and, for each $C\in\mathcal A\cup\mathcal B$, satisfying $(0,C)\preceq (0,B)$.

We use \cite[Lemma 2.3]{SS} again, now for the module $\Phi(A,B)$, to find an upward directed set $V(\lambda)\subseteq\mathcal W\cap(\mathcal A\times\mathcal B)$ which satisfies the hypotheses of Axiom $(4)$.

Axiom $(5)$ is just the $3\times3$ lemma applied, for each $D\in\mathcal D$, on the $\Hom_R(-,D)$-image of the following commutative diagram with exact rows and columns

$$\begin{CD}
 0  @. 0  @.   0 \\
  @AAA				@AAA		@AAA \\
	\Phi(A,B) 	@>>>		\Phi(C,D) @>>> 	\Coker(f)		@.	\\
	@AAA				@AAA	@AAA		\\
	\Phi(0,B) @>>>	\Phi(0,D) @>>>	 	\Phi(B,D)	@>>>	0	\\
	@AAA				@AAA	@A{f}AA		\\
	\Phi(0,A) 	@>>>	\Phi(0,C)	@>>> 	\Phi(A,C)	@>>>	0.
\end{CD}$$

\smallskip

For the Axiom $(6)$, let $g\in \Hom_R(\Phi(0,A_\omega),D)$ be arbitrary with $D\in\mathcal D$. Notice that $\Phi(A_k,B_k)\in {}^\perp\mathcal D$ for each $k<\omega$ since $\Phi(0,B_k)\in {}^\perp\mathcal D$ and $(0,A_k)\preceq (0,B_k)$.

Using Lemma~\ref{l:eklof}, it follows that $\Phi(A_\omega,B_\omega)$ belongs to ${}^\perp \mathcal D$, so it remains to check that $\Ker(\Phi(0,A_\omega) \to \Phi(0,B_\omega))\subseteq\Ker(g)$. Assume that $x\in\Phi(0,A_\omega)$ is arbitrary such that $g(x)\neq 0$. There exists $j<\omega$ such that a~preimage $y$ of $x$ can be found in $\Phi(0,A_j)$. Since $(0,A_j)\preceq (0,B_\omega)$, we infer that $y\not\in\Ker(\Phi(0,A_j) \to \Phi(0,B_\omega))$, whence $x\not\in\Ker(\Phi(0,A_\omega) \to \Phi(0,B_\omega))$.

\smallskip

For the `furthermore' case, we can, using Construction~\ref{conmerge}, w.l.o.g.\ assume that, for each successor cardinal $\lambda>\nu$, the system $\mathcal S_\lambda$ is a subsystem of $\mathcal S_\eta^\lambda$ for any successor cardinal $\nu<\eta<\lambda$. It follows that Axiom $(4)$ from Definition~\ref{d:axioms} for $A = 0$ holds in this case, which allows us to apply Proposition~\ref{p:lite}.
\end{proof}

\appendix

\section{Manipulating with directed systems}
\label{sec:append}

In the whole paper, we often use the following two constructions of directed systems of modules (or morphisms).

\begin{constr} \label{concoker} Let $f:M \to N$ be a homomorphism of modules, $\lambda$ a regular uncountable cardinal, $\mathcal M = (M_i, f_{ji}:M_i \to M_j\mid i<j\in I)$ and $\mathcal N = (N_i, g_{ji}:N_i \to N_j\mid i< j\in J)$ $\lambda$-continuous directed systems of $<\lambda$-presented modules \st $\varinjlim\mathcal M = M$ and $\varinjlim\mathcal N = N$. Then there is a $\lambda$-continuous directed system $\mathcal U = (u_k:M_{i_k}\to N_{j_k}, (f_{i_l,i_k},g_{j_l,j_k})\mid k<l\in K)$ consisting of morphisms with domains in $\mathcal M$ and codomains in $\mathcal N$ such that $\varinjlim\mathcal U = f$.

Subsequently, there is a $\lambda$-continuous directed system $\mathcal K = (\Coker(u_k)\mid k\in K)$ consisting of $<\lambda$-presented modules $($with canonically defined maps$)$.
\end{constr}

\begin{proof} For each $i\in I$ and $j\in J$, let us denote by $f_i:M_i\to M$ and $g_j:N_j\to N$ the colimit maps, and define $f_{ii} = \hbox{id}_{M_i}$ and $g_{jj} = \hbox{id}_{N_j}$.

We define $\mathcal U$ as the set of all morphisms $u:M_i \to N_j$ such that $i\in I, j\in J$, $g_ju = ff_i$. For $u:M_i\to N_j,v:M_r \to N_s$ from $\mathcal U$, we put $u\leq v$ \iff $i\leq r, j\leq s$ and $vf_{ri} = g_{sj}u$. We easily check that $(\mathcal U,\leq)$ is a poset. Next, we show that it is directed.

First, fix generating sets $G = \{x_\alpha\mid \alpha<\mu\}$ and $H = \{y_\alpha \mid \alpha< \mu\}$ of $M_i$ and $M_r$, respectively, where $\mu <\lambda$, and let $u,v\in\mathcal U$ be as above. We find $a\in I$ \st $i,r<a$. Since $\mathcal N$ is $\lambda$-continuous and $M_a$ is $<\lambda$-presented, there is $b\in J, j,s<b$, and a morphism $w_0:M_a\to N_b$ from $\mathcal U$. It need not be the case that $u\leq w_0$ and $v\leq w_0$, however, for each $\alpha<\mu$, there is a $b_\alpha\geq b$ such that $g_{b_\alpha j}u(x_\alpha) = g_{b_\alpha b}w_0f_{ai}(x_\alpha)$ and $g_{b_\alpha s}v(y_\alpha) = g_{b_\alpha b}w_0f_{ar}(y_\alpha)$. Since $\mu<\lambda$ and $(J,\leq)$ is $\lambda$-directed, there is $c\in J$ such that $c\geq b_\alpha$ for each $\alpha<\mu$. It follows that $w = g_{cb}w_0$ is in $\mathcal U$ and $u,v\leq w$. Subsequently, $(\mathcal U,\leq)$ is a directed system of morphisms. Moreover, it is $\lambda$-continuous since $\mathcal M$ and $\mathcal N$ are such.

To prove that $\varinjlim \mathcal U = f$, it is now enough to find, for arbitrary $(i,j)\in I\times J$, a morphism $u:M_i \to N_s$ in $\mathcal U$ with $s\geq j$. This is easy (recall how we found $w_0$).
\end{proof}

The next tool allows us to merge less than $\lambda$ directed systems which are $\lambda$-continuous into one. In its statement, we do not use the notation from Definition~\ref{d:inducesys}.

\begin{constr} \label{conmerge} Let $M\in\ModR$, $\lambda$ be a regular uncountable cardinal, $\gamma<\lambda$ and, for each $\alpha\leq\gamma$, let $\mathcal M^\alpha = (M^\alpha_{ji},f^\alpha_{ji}\colon M^\alpha_i\to M^\alpha_j\mid i<j\in I_\alpha)$ be a $\lambda$-continuous directed system consisting of $<\lambda$-presented modules such that $\varinjlim\mathcal M^\alpha = M$. Then the systems $\mathcal M^\alpha$, $\alpha\leq\gamma$, have a common cofinal $\lambda$-continuous subsystem.

More precisely:
for each $\alpha\leq\gamma$, there exists a $\lambda$-continuous cofinal directed subsystem $\mathcal N^\alpha = (M^\alpha_{ji},f^\alpha_{ji}\colon M^\alpha_i\to M^\alpha_j\mid i<j\in J_\alpha)$ of $\mathcal M^\alpha$; furthermore, for any $\alpha,\beta\leq\gamma$, there is a bijection $\iota\colon J_\alpha \to J_\beta$ and a directed system $\mathcal U$ with $\varinjlim \mathcal U = \hbox{id}_M$ whose objects are isomorphisms $u_i\colon M^\alpha_i \to M^\beta_{\iota(i)}$, $i\in J_\alpha$, and for each $i<j\in J_\alpha$, there is only one morphism from $u_i$ to $u_j$ in $\mathcal U$, namely $(f^\alpha_{ji},f^\beta_{\iota(j),\iota(i)})$.
\end{constr}

\begin{proof} We can assume that $\gamma$ is a cardinal. The proof goes by induction on $\gamma$. For $\gamma = 0$, it is trivial. Let $\gamma = 1$.

We use Construction~\ref{concoker} with $f = \hbox{id}_M$, $\mathcal M = \mathcal M^0$ and $\mathcal N = \mathcal M^1$ to obtain the system $\mathcal U$ of morphisms. 
Using \cite[Lemma 2.6]{BR}, we can w.l.o.g.\ assume that the objects of $\mathcal U$ are isomorphisms. The subsystems $\mathcal N^\alpha$, $\alpha = 0, 1$, then consist of domains, codomains, respectively, of the isomorphisms in $\mathcal U$.

By induction, we have the proof for any $\gamma$ finite. For $\gamma$ infinite, we use the inductive hypothesis and the following simple fact: for each $\alpha\leq \gamma$, if $(\mathcal N^\alpha_\beta \mid \beta<\gamma)$ is a family of $\lambda$-continuous cofinal directed subsystems of $\mathcal M^\alpha$ \st $\mathcal N^\alpha_\beta\supseteq \mathcal N^\alpha_\delta$ whenever $\beta\leq\delta<\gamma$, then $\bigcap _{\beta<\gamma} \mathcal N^\alpha_\beta$ is a $\lambda$-continuous cofinal directed subsystem of $\mathcal M^\alpha$ as well.
\end{proof}

We also use freely the following easy

\begin{obser} \label{o:observ} Let $\lambda$ be a regular uncountable cardinal and $\mathcal M$ a $\lambda$-continuous directed system of modules. Let $\mathcal K$ be a directed subsystem of $\mathcal M$. Then there is a~$\lambda$-continuous directed subsystem $\mathcal K^\prime$ of $\mathcal M$ with the same direct limit as $\mathcal K$.
\end{obser}


\end{document}